\documentclass[a4paper]{article}

\newcommand\stypath{.}

\usepackage[russian,english]{babel}

\usepackage[pdffitwindow=false,
            plainpages=false,
            pdfpagelabels=true,
            pdfpagemode=UseOutlines,
            pdfpagelayout=SinglePage,
            bookmarks=false,
            colorlinks=true,
            hyperfootnotes=false,
            linkcolor=blue,
            urlcolor=blue!30!black,
            citecolor=green!50!black]{hyperref}

\usepackage{\stypath /common}
\usepackage{\stypath /kvn}

\usepackage{\stypath /review}

\usepackage{microtype}

\usepackage[backend=biber,style=alphabetic]{biblatex}
\usepackage{geometry}

\usepackage{authblk}

\usepackage{chngcntr}

\usepackage{nicefrac}

\newcommand{\pd}[2]{\frac{\partial#1}{\partial#2}}

\usepackage[normalem]{ulem}
\usepackage{todonotes}

\title{Existence and Uniqueness of Solutions of the Koopman--von Neumann Equation on Bounded Domains}

\author[1,2]{Marian Stengl\thanks{corresponding author, e-mail address: \texttt{stengl@math.tu-berlin.de}}}
\author[1,3]{Patrick Gel\ss}
\author[4]{Stefan Klus}
\author[1,2]{Sebastian Pokutta}

\affil[1]{AI in Society, Science, and Technology, Zuse Institute Berlin, Berlin 14195, Germany}

\affil[2]{Institute of Mathematics, Technische Universit\"at Berlin, Berlin 10623, Germany}

\affil[3]{Institute of Mathematics, Freie Universit\"at Berlin, Berlin 14195, Germany}

\affil[f]{School of Mathematical \& Computer Sciences, Heriot--Watt University, Edinburgh EH14 4AS, UK}

\begin{document}
%
%
\maketitle%
%
%
%

\abstract{%
	The Koopman--von Neumann equation describes the
	evolution of a complex-valued wavefunction corresponding to the probability distribution given by an associated classical Liouville equation. Typically, it is defined on the whole Euclidean space.
	The investigation of bounded domains, particularly in practical scenarios involving quantum-based simulations of dynamical systems, has received little attention so far.
	We consider the Koopman--von Neumann equation associated with an ordinary differential equation on a bounded domain whose trajectories are contained in the set's closure.
	Our main results are the construction of a strongly continuous semigroup together with the existence and uniqueness of solutions of the associated initial value problem.
	To this end, a functional-analytic framework connected to Sobolev spaces is proposed and analyzed.
	Moreover, the connection of the Koopman--von Neumann framework to transport equations is highlighted.
}
~\\[0.1cm]
\noindent \begin{minipage}[t]{0.15\linewidth}
	\textbf{Keywords:}
\end{minipage}%
\begin{minipage}[t]{0.85\linewidth}
	dynamical systems, transfer operators, evolution equations, Koopman--von Neumann mechanics, Perron--Frobenius--Sobolev space
\end{minipage}

~\\

\noindent \begin{minipage}[t]{0.15\linewidth}
	\textbf{MSC:}
\end{minipage}%
\begin{minipage}[t]{0.85\linewidth}
	35A05, 35F10, 37C30, 46E35, 47D06
\end{minipage}






%
%
%
\section{Introduction}\label{sec:intro}
%
%
Quantum computing has the potential to enhance the way information is processed and enables us to solve problems that are beyond the capabilities of classical computers.
Since the introduction of this paradigm by Benioff~\cite{Benioff1980}, Manin~\cite{Manin1980}, and in particular Feynman~\cite{Feynman1982} in the early 1980s, the interest in quantum computation and simulation has been growing continuously.
Recent years have seen rapid progress not only in terms of technical realizations but also by opening up new applications including cryptography~\cite{Pirandola2020, Portmann2022}, financial modeling~\cite{Coyle2021, Stamatopoulos2022}, materials science~\cite{Bauer2020, Schlawin2022}, and machine learning~\cite{Zhang2020, Cerezo2022}.

%
%
Quantum computers are also of interest to simulate dynamical systems modeled by ordinary differential equations (ODEs).
This problem class covers a variety of complex physical, chemical, and biological systems.
As quantum algorithms are written in terms of unitary operations, the realization of such an algorithm is a challenging task.
While the solution of high-dimensional systems of \emph{linear} ODEs has been extensively studied and various quantum algorithms have been developed in the past decades~\cite{Lloyd1996, Aharonov2003, Berry2007, Harrow2009, Okounkov2010, Berry2017, Childs2019, Zanger2021}, the simulation of \emph{nonlinear} ODEs is a more challenging task that has only been addressed recently~\cite{Kyriienko2021, Liu2021, Xue2021, Krovi2023, Shukla2023}.

%
%
One possible approach towards the analysis of dynamical systems is the use of \emph{transfer operators}.
In this context, the Perron--Frobenius and the Koopman operators~\cite{Ko31, LaMa94, DJ99, Mezic05, Giannakis2022} are of particular interest, as they enable the analysis of the global behavior of complex dynamical systems.
Instead of solving nonlinear ODEs through numerical integration, transfer operators describe how probability densities and observables, respectively, are propagated forward in time by using linear and infinite-dimensional operators and their corresponding generators.
Transfer operator theory has been successfully applied in various research areas including fluid dynamics \cite{Mezic05, RMBSH09}, molecular dynamics \cite{Sch99, NKPMN14, KKS16, SKH23}, and ergodic theory~\cite{Ko31, Fiedler2001, EFHN2001, AM2017, GIANNAKIS2019}.
It lays the foundation for a wide range of data-driven methods which can be used, e.g., for the detection of metastable or coherent sets as well as model reduction~\cite{KKS16, KNKWKSN18, KoMe18, KNPNCS20, KKB21}, but also for control~\cite{KM18a, PK19, POR20, KKB21, KNP22}.

%
%
For the operator-based numerical simulation and analysis of dynamical systems on a quantum computer, the \emph{Koopman--von Neumann} (KvN) framework is of particular interest.
In this formulation, the conservation of the probability distribution, expressed by the classical Liouville equation, is recast as a Schrödinger equation on a Hilbert space.
The solution of this equation yields a complex-valued wave function $\psi$, whose propagation is described by a semigroup of unitary operators.
Using Born's rule, one can extract the probability density $\rho = |\psi|^2$.
This probability density then satisfies the Liouville equation, see \cite{Joseph20,bib:BogdanovBognanova,bib:BoganovBogdanovaFastovetsLukichev}.
In its original formulation, Koopman and von Neumann only considered the case of Hamiltonian dynamics~\cite{Ko31, bib:vonNeumann_Koopman}, but it has been extended to general dynamical systems, see \cite{Joseph20,bib:BogdanovBognanova,bib:BoganovBogdanovaFastovetsLukichev}.
The KvN framework thus establishes a connection between classical mechanics and quantum mechanics.
See also \cite{LLASS22} for a detailed comparison of the different operators and their properties.
If the KvN Hamiltonian is sparse, the quantum-based simulation is exponentially more efficient than the Euler discretization of the Liouville equation~\cite{Joseph20}.
All these considerations motivate an in-depth analysis of the mathematical theory behind the KvN equation.

%
%
\revb{%
The primary goal of this paper is to mathematically investigate the KvN framework.
We focus on dynamical systems with trajectories fully contained within a bounded, closed, and sufficiently regular set for all times.
For these systems, we will derive a rigorous existence theory for the KvN equation.
The idea behind this paper has grown from two aspects:
\begin{enumerate}
    \item There was in our opinion a gap in both, the mathematical and physical literature, regarding the existence of solutions. This needed to be addressed very early to lay a mathematical foundation.
    \item All the literature we found was taking into account the Koopman--von Neumann framework on the whole space $\bR^d$. However, in particular with regard to quantum computing bounded domains are of interest as an unbounded domain is hard to discretize.
\end{enumerate}
The key requirement for the existence of solutions is the no-outflow condition on the right-hand side of the given ODE.
This condition ensures that trajectories remain within the bounded domain for all times.

Recognizing the dual interest from both communities, we provide a comprehensive collection of the tools used in this study.%
}{}
To this end, key concepts such as semigroup theory, function spaces, transfer operators and their connections to the KvN framework will be introduced in detail and many additional references are provided for the interested reader.
We will in particular highlight the relationship with the transport equation literature.
%
%

The rest of this paper is organized as follows:
In \cref{sec:np} we will introduce selected tools from functional analysis including semigroup theory in \cref{ssec:semigroup}, function spaces in \cref{ssec:function_spaces}, and transfer operators in \cref{ssec:derivation_kvn}.
A function space framework related to Sobolev spaces will be introduced and analyzed in \cref{sec:pfs}. In \cref{sec:existence_kvn} these results will be used to derive the existence of a strongly continuous semigroup associated with the KvN generator and hence to prove the existence and uniqueness of solutions.

%
%
\section{Notation and Preliminaries}\label{sec:np}
For the following abstract notions in functional analysis, we refer to \cite{bib:AltNuernberg}.
Let $X$ be a real (or complex) Banach space.
Its \emph{(topological) dual space} is denoted by $X^*$ and is defined as the set of all bounded linear operators taking values in $\bR$ or $\bC$.
The \emph{dual pairing} of elements $x^* \in X^*$ and $x \in X$ is defined by $\langle x^*, x \rangle_{X^*, X} := x^*(x)$.
If the corresponding spaces are clear from the context, we may just write $\langle x^*, x \rangle$.
Let $Y$ be another (real or complex) Banach space.
The set of all bounded linear operator from $X$ to $Y$ is denoted by $\cL(X,Y)$ and forms a Banach space when equipped with the \emph{operator norm}
\begin{eq*}
	\|A\|_{\cL(X,Y)} := \sup_{\substack{x \in X \\ \|x\|_X = 1}} \|A x\|_Y
\end{eq*}
for $A \in \cL(X,Y)$.
Clearly, we have $X^* = \cL(X,\bR)$ for real, and $X^* = \cL(X,\bC)$ for complex Banach spaces.
If $Y = X$ we will sometimes simply write $\cL(X) := \cL(X,X)$.
The identity on $X$ is denoted by $\id_X$.
The \emph{dual} of an operator $A \in \cL(X,Y)$ is the uniquely determined operator $A^* \in \cL(Y^*, X^*)$ with
\begin{eq*}
	\langle A^* y^* , x \rangle_{X^*, X} = \langle y^*, Ax \rangle_{Y^*,Y}, \quad \text{for all~} x \in X, y^*\in Y^*.
\end{eq*}

Let $Y \subseteq X$ be a subspace equipped with its own norm $\|\cdot\|_Y$.
If $(Y, \|\cdot\|_Y)$ is a real (respectively complex) Banach space and there exists a constant $C > 0$ with $\|y\|_X \leq C \|y\|_Y$ for all $y \in Y$, then we say that $Y$ \emph{embeds continuously} into $X$ and we write $Y \embeds X$.

Finally, let $H$ be a complex Hilbert space with scalar product $(\cdot,\cdot)_H$.
Then, according to the Riesz representation theorem~\cite[Section 6.1]{bib:AltNuernberg}, the canonical map $\Lambda \colon H \to H^*$ defined by ${u \mapsto (v \mapsto (u,v)_H)}$ is an isometric, conjugate linear map, i.e., $\Lambda(u_0 + u_1) = \Lambda u_0 + \Lambda u_1$ and $\Lambda(\alpha u) = \bar\alpha \Lambda u$ for all $u, u_0, u_1 \in H$ and $\alpha \in \bC$.
%
%
\subsection{Semigroup Theory}\label{ssec:semigroup}
\revb{%
In this section, we introduce the concepts and results from semigroup theory used in the main part of the article.
For more details, see \cite{bib:Pazy_Semigroup,bib:EngelNagel}.
Our primary goal is to prove the existence of solutions for the KvN equation, which can be classified as a linear evolution equation.
The mathematical framework for deriving the existence of solutions is \emph{semigroup theory}.
For finite-dimensional linear ODEs, solutions are characterized using the matrix exponential.
Semigroups generalize this concept to linear operators on infinite-dimensional Banach spaces.

}%
A \emph{one-parameter semigroup} (on $X$) or just \emph{semigroup} is a family of operators $(T(t))_{t \geq 0} \subseteq \cL(X)$ such that the following properties hold:
\begin{enum}
	\item $T(0) = \id_X$,
	\item $T(s + t) = T(s)T(t)$ for all $s, t \geq 0$.
\end{enum}
A semigroup is called a \emph{$C_0$-semigroup} if for all $x \in X$ the mapping $t \mapsto T(t)x$ is continuous on $[0,\infty)$ and it is called a semigroup of \emph{contractions}, if it holds
that $\|T(t)\|_{\cL(X)} \leq 1$ for all $t \geq 0$.
The \emph{(infinitesimal) generator} of $(T(t))_{t \geq 0}$ is defined by
\begin{eq*}
	Ax := \lim_{t \searrow 0} \frac{1}{t} (T(t) - \id_X)x
\end{eq*}
for all $x \in \cD(A)$ with
\begin{eq*}
	\cD(A) := \left\{ x \in X : \lim_{t \searrow 0} \frac{1}{t} (T(t) - \id_X)x \text{~exists} \right\}.
\end{eq*}
This set is called the \emph{domain} of $A$.
In general, $\cD(A)$ is not a closed subspace of $X$ and $A$ is an \emph{unbounded operator}.
Such an operator is called \emph{closed}, if its graph
\begin{eq*}
	\gph{A} := \{ (x,y) \in X \times X : x \in \cD(A) \text{~and~} y = Ax \}
\end{eq*}
is a closed subset of $X \times X$.
If $A \colon \cD(A) \subseteq X \to X$ is a closed operator, then the space $\cD(A)$ equipped with the \emph{graph norm}
\begin{eq*}
	\|x\|_A := \|x\|_X + \|A x\|_X
\end{eq*}
is a Banach space and $A \in \cL(\cD(A), X)$ is a bounded linear operator.
For a complex Hilbert space $H$, an operator $A \colon \cD(A) \subseteq H \to H$ is called \emph{dissipative}, if all $u \in \cD(A)$ satisfy $\re (Au, u)_H \leq 0$.

Consider the following \emph{Cauchy problem}
\begin{eq}\label{eq:evo}
	\partial_t u &= A u \text{~for all~} t \geq 0,\\
	u(0) &= u_0.
\end{eq}
Of particular interest for us is the following result.
\begin{thm}[Solution of the Cauchy problem]\label{thm:evo_existence_uniqueness}
	Let $A \colon \cD(A) \subseteq X \to X$ be the infinitesimal generator of a $C_0$-semigroup $(T(t))_{t \geq 0}$ with $T(t) \in \cL(X)$ for all $t \geq 0$, then the mapping $u\colon t \mapsto T(t)u_0$ is the unique solution $u \in C^1([0,\infty), X) \cap C([0,\infty), \cD(A))$ of \cref{eq:evo}.
\end{thm}

For more details, see, e.g., \cite[Chapter 4, Theorem 1.3]{bib:Pazy_Semigroup} as well as	\cite[Proposition 2.2.2(i)]{bib:LorenziRhandi_Semigroup}. The existence and uniqueness of solutions of \cref{eq:evo} can thus be guaranteed if the existence of a $C_0$-semigroup for a $A \colon \cD(A) \subseteq X \to X$ can be established.
There are several results that guarantee the existence of a $C_0$-semigroup for a given unbounded operator $A$ with domain $\cD(A)$, e.g., the Hille--Yosida theorem~\cite[Chapter 2, Generation Theorem 3.5]{bib:EngelNagel} and Lumer--Phillips theorem~\cite[Chapter 2, Theorem 3.15]{bib:EngelNagel}.
For our case, the following corollary suffices.
%
%
\begin{cor}[Corollary of Lumer--Phillips Theorem, see \mbox{\cite[Corollary 4.4]{bib:Pazy_Semigroup}}]\label{cor:lumer_phillips_cor}
	Let $H$ be a complex Hilbert space and $A$ a densely defined closed linear operator on $H$.
	If both $A$ and $A^*$ are dissipative, then $A$ is the infinitesimal generator of a $C_0$-semigroup of contractions on $H$.
\end{cor}
%
%
%
%
\subsection{Function Spaces}\label{ssec:function_spaces}
\revb{%
As previously noted, we consider an evolution equation in an infinite-dimensional setting.
Since the evolved KvN operator involves partial differential operators, we have to deal with function spaces that resolve different types of differentiability.%
}{}
We will frequently use several results regarding function spaces, which are introduced in this subsection.
For further details, we refer to~\cite{bib:AltNuernberg,bib:EvansGariepy,bib:AdamsFournier} as well as to~\cite[Chapter 1, §2]{bib:GiraultRaviart_FEMNavierStokes}.

Throughout this work, let $\Omega \subseteq \bR^d$, $d \in \bN$, be a bounded, open domain with \emph{Lipschitz boundary}.
The latter means that for every $x \in \partial\Omega$ there exists a neighborhood $Q$ such that after a change of coordinates the intersection $Q \cap \partial\Omega$ can be identified with the graph of a Lipschitz continuous function on a neighborhood of $\bR^{d - 1}$.
By $\lambda^d$ we denote the \emph{Lebesgue measure} and by $\cH^{d-1}$ the $(d - 1)$-dimensional Hausdorff measure.
Since $\Omega$ has a Lipschitz boundary, the \emph{outward unit normal} $\nu \colon \partial\Omega \to \bR^d$ is well-defined up to a Hausdorff null set.

%
%
\subsubsection{Continuous and Differentiable Functions}

The set of all \emph{continuous} functions on $\Omega$ is denoted by $C(\Omega)$ and the set of all such functions that are \emph{continuous to the boundary} is denoted by $C(\bar\Omega)$.
As $\Omega \subseteq \bR^d$ is bounded, its closure is compact.
Hence, all functions in $C(\bar\Omega)$ are bounded and uniformly continuous on $\bar\Omega$.
For $k \in \bN$, $k \geq 1$, let $C^k(\Omega)$ and $C^k(\bar\Omega)$ denote the space of all \emph{$k$-times differentiable} functions such that all partial derivatives of order up to $k$ are in $C(\Omega)$ or $C(\bar\Omega)$, respectively.
For $k = 0$, we define $C^0(\Omega) = C(\Omega)$ (and $C^0(\bar \Omega) = C(\bar \Omega)$).
The spaces
\begin{eq*}
	C^\infty(\Omega) := \bigcap_{k \in \bN} C^k(\Omega) \text{~and~} C^\infty(\bar\Omega) := \bigcap_{k \in \bN} C^k(\bar\Omega)
\end{eq*}
are the sets of \emph{smooth functions} on $\Omega$ or $\bar\Omega$, respectively.
The \emph{support} of a function $u \colon \Omega \to \bR$ is defined as
\begin{eq*}
	\supp u := \overline{\{ x \in \Omega : u(x) \neq 0\}},
\end{eq*}
where the closure is taken in $\bR^d$.
A function is said to be \emph{compactly supported} (in $\Omega$), if its support is contained in $\Omega$.
The set of all \emph{smooth compactly supported} functions is denoted by $C^\infty_0(\Omega)$.
For a Banach space $X$, we analogously define the spaces $C^k(\Omega, X)$ and $C^\infty_{(0)}(\Omega, X)$ as the set of all $k$-times differentiable and smooth (as well as compactly supported) functions, respectively, on $\Omega$ with values in $X$.

%
%
\subsubsection{Lebesgue Spaces}
In the remainder of the text, we always identify a function $u \colon \Omega \to \bR$ with its equivalence class $[u] = \left\{ v \colon \Omega \to \bR \,\colon\, \lambda^d(\{x \in \Omega : u(x) \neq v(x)\}) = 0 \right\}$ and write $\mathrm{d} x = \mathrm{d} \lambda^d (x)$ as a shorthand.
Given the measure space $(\Omega, \cL(\Omega), \lambda^d)$ with Lebesgue $\sigma$-algebra $\cL(\Omega)$, the $L^p$ space for $p \in \left[1, \infty \right)$ can be defined as the set of equivalence classes of Lebesgue measurable functions
\begin{eq*}
L^p(\Omega) = \left\{   u \colon \Omega \to \bR \,\colon\, \int_{\Omega} |u|^p \mathrm{d}x < \infty \right\}.
\end{eq*}
For $p = \infty$, we define the space
\begin{eq*}
	L^\infty(\Omega) := \{ u \colon \Omega \to \bR  \,\colon\, \exists\,C > 0 \text{~with~} |u| \leq C \text{~$\lambda^d$-a.e.~on~} \Omega\}.
\end{eq*}
The above spaces are Banach spaces when equipped with the norms
\begin{eq*}
	\|u\|_{L^p(\Omega)} := \left( \int_\Omega |u(x)|^p \mathrm{d}x \right)^{1/p} \text{for~}p \in [1, \infty) \text{~and~}
	\|u\|_{L^\infty(\Omega)} := \inf_{\substack{N \subseteq \Omega,\\ \lambda^d(N) = 0}} \sup_{x \in \Omega \backslash N} |u(x)|,
\end{eq*}
respectively.
For $p = 2$, the norm is induced by the inner product
\begin{eq*}
	(u,v)_{L^2(\Omega)} := \int_\Omega u  v \,\dx, \quad \text{for all~} u, v \in L^2(\Omega)
\end{eq*}
and therefore $L^2(\Omega)$ is a Hilbert space.
The space $L^p(\Omega, \bR^{m})$, $m \in \bN$, is the set of functions ${u \colon \Omega \to \bR^{m}}$ whose components are in $L^p(\Omega)$.
For $p,q \in (1, \infty)$ with $\frac{1}{p} + \frac{1}{q} = 1$, the dual of $L^p(\Omega)$ is isomorphic to $L^q(\Omega)$.
Moreover, we have $(L^1(\Omega))^* = L^\infty(\Omega)$, but in general $(L^\infty(\Omega))^* \neq L^1(\Omega)$, see \cite[Theorems 2.44--2.46, Remark 2.47]{bib:AdamsFournier}.

%
%
\subsubsection{\texorpdfstring{Sobolev Space and $H(\ddiv)$ Space}{Sobolev Space and H(div) Space}}

The \emph{distributional derivative} of $u \in L^p(\Omega)$ (with respect to the $j$-th coordinate) is the linear functional $\partial_j u \colon C_0^\infty(\Omega) \to \bR$ defined by
\begin{eq*}
	\langle \partial_j u, \varphi \rangle
	:= - \int_\Omega u \, \partial_j \varphi \,\dx
	~= - (u, \partial_j \varphi)_{L^2(\Omega)}
\end{eq*}
for $j = 1, \dots, d$ and for all $\varphi \in C_0^\infty(\Omega)$.
(Here, the use of the dual pairing is some slight abuse of notation.)
By definition, the distributional derivative \emph{always} exists.
%
%
The Sobolev space (for $p = 2$) is defined as
\begin{eq*}
	H^1(\Omega) := \left\{ u \in L^2(\Omega) : \partial_j u \in L^2(\Omega) \text{~for all~} j = 1, \dots, d\right\}
\end{eq*}
and is a Hilbert space when equipped with the inner product
\begin{eq*}
	(u,v)_{H^1(\Omega)} := (u,v)_{L^2(\Omega)} + (\nabla u, \nabla v)_{L^2(\Omega;\bR^d)} \text{~for all~} u, v \in H^1(\Omega).
\end{eq*}
In this case, the distributional derivative is called \emph{weak derivative}.
Sobolev functions exhibit a range of practical properties, one of which is a simple product rule highlighted in the following theorem.

\begin{thm}[Product rule, \mbox{cf.~\cite[Sec.~4.25]{bib:AltNuernberg}}]\label{thm:sobolev:product_rule}
	\reva{For}{} two functions $u\in H^1(\Omega)$ and $f \in C^1(\bar \Omega)$ \reva{we have $u f \in H^1(\Omega)$ with}{}
	\begin{eq*}
		\nabla (u f) = f \nabla u + u \nabla f \in L^2(\Omega;\bR^d).
	\end{eq*}
\end{thm}

There exist other versions of the product rule for Sobolev functions, but within the scope of this work \cref{thm:sobolev:product_rule} suffices.
Another important feature of Sobolev functions pertaining to boundary conditions is the existence of \emph{traces}.
This is covered in the next theorem.
%
%
\begin{thm}[Traces of Sobolev functions, cf.~\mbox{\cite[Theorem 4.6]{bib:EvansGariepy}}]\label{thm:trace_sobolev}
	Let $\Omega \subseteq \bR^d$ be a bounded, open domain with Lipschitz boundary.
	Then, there exists a linear, continuous operator $\tr: H^1(\Omega) \to L^2(\partial\Omega)$ that satisfies $\tr(u) = u|_{\partial\Omega}$ for all $u \in C^1(\bar \Omega)$.
\end{thm}

Of particular importance are functions that are zero on the boundary.
This space is defined by $H^1_0(\Omega) = \ker{\tr}$.
Under the conditions of \cref{thm:trace_sobolev}, one can show that this space is the completion of $C_0^\infty(\Omega)$ with respect to the $H^1$-norm, cf. \cite[Theorem A8.10 and 3.29]{bib:AltNuernberg}.
Its dual space is referred to as $H^{-1}(\Omega)$.
The trace operator is \emph{not} surjective and its range is the space $H^\frac{1}{2}(\partial \Omega)$ equipped with the quotient norm
\begin{eq*}
	\|z\|_{H^{1/2}(\partial \Omega)} = \inf_{\substack{u \in H^1(\Omega),\\ \tr(u) = z}} \|u\|_{H^1(\Omega)}.
\end{eq*}
The dual of this space is denoted by $H^{-1/2}(\partial \Omega)$.

%
%
Similar to the definition of the distributional derivative, we define the \emph{distributional divergence} of a vector field $u \in L^p(\Omega; \bR^d)$ by
\begin{eq}\label{eq:distributional_divergence}
	\langle \ddiv u, \varphi \rangle
	:= - \int_\Omega u \cdot \nabla \varphi \, \dx.
\end{eq}%
Analogously, we are then interested in the existence of \emph{weak divergences}.
This motivates the definition of the space
\begin{eq*}
	H(\ddiv,\Omega) := \left\{u \in L^2(\Omega;\bR^d) : \ddiv u \in L^2(\Omega) \right\}.
\end{eq*}
Hence, we have $u \in H(\ddiv, \Omega)$ if and only if $u \in L^2(\Omega;\bR^d)$ and there exists a $y \in L^2(\Omega)$ such that
\begin{eq*}
	(u, \nabla \varphi)_{L^2(\Omega;\bR^d)} = -(y, \varphi)_{L^2(\Omega)} \text{~for all~} \varphi \in C_0^\infty(\Omega).
\end{eq*}
This space is a Hilbert space when equipped with the inner product
\begin{eq*}
	(u,v)_{H(\ddiv,\Omega)} := (u,v)_{L^2(\Omega;\bR^d)} + (\ddiv(u), \ddiv(v))_{L^2(\Omega)}.
\end{eq*}
It can then be shown that a function whose components are in $H^1(\Omega)$ is indeed in $H(\ddiv,\Omega)$.
Returning to the product rule described in \cref{thm:sobolev:product_rule}, we would like to highlight the following frequently encountered case:
Take a function $F \in C^1(\bar\Omega;\bR^d)$ with entries $F_j \in C^1(\bar\Omega)$.
By \cref{thm:sobolev:product_rule}, all components of $u F$ are in $H^1(\Omega)$ and its weak divergence reads
\begin{eq}\label{eq:product_rule}
	\ddiv(u F) = \sum_{j = 1}^d \partial_j(u F_j) = F \cdot \nabla u + (\ddiv F) u.
\end{eq}
We will also refer to \cref{eq:product_rule} as a \emph{product rule}.

Analogously to the Sobolev space case, one can ask for an extension of the trace.
It turns out that it only makes sense to ask for an extension of the \emph{trace in normal direction}.
%
%
\begin{thm}[Traces of $H(\ddiv)$ functions, see \reva{\mbox{\cite[Theorem 2.5]{bib:GiraultRaviart_FEMNavierStokes}}}]\label{thm:trace_H_div}
	Let $\Omega \subseteq \bR^d$ be a bounded, open domain with Lipschitz boundary.
	\reva{Then,}{} there exists a bounded linear operator $\tr_\nu \colon H(\ddiv,\Omega) \to H^{-1/2}(\partial \Omega)$ such that for all $p \in C^1(\bar\Omega; \bR^d)$ \reva{holds}{} $\tr_\nu(p) = p|_{\partial\Omega} \cdot \nu$, where $\nu \colon \partial \Omega \to \bR^d$ denotes the outward unit normal of $\Omega$.
\end{thm}
One should note that the values of this trace are of fairly low regularity as they are only distributions in general.
An interesting feature is the extension of Green's formula.
%
%
\begin{thm}[Green's formula, see \reva{\mbox{\cite[eq. (2.17)]{bib:GiraultRaviart_FEMNavierStokes}}}]\label{thm:green_Sobolev_H_div}
	Let $\Omega \subseteq \bR^d$ be a bounded, open domain with Lipschitz boundary.
	Then, \reva{for all $u \in H^1(\Omega)$ and $p \in H(\ddiv,\Omega)$ we have}{}
	\begin{eq*}
		\langle \tr_\nu p, \tr u \rangle_{H^{-1/2}(\partial\Omega), H^{1/2}(\partial\Omega)} = (\ddiv p, u)_{L^2(\Omega)} + (p, \nabla u)_{L^2(\Omega;\bR^d)}.
	\end{eq*}
\end{thm}
\cref{thm:green_Sobolev_H_div} facilitates an extension of Green's formula from $C^1$ functions to Sobolev and $H(\ddiv)$ functions.
Moreover, it involves the normal trace, which is in general only a distribution.
%
%
\subsection{Derivation of the Koopman--von Neumann Generator}
\label{ssec:derivation_kvn}
In this subsection, we return to the ODE setting from the introduction.
Consider the following autonomous initial value problem (IVP)
\begin{eq}\label{eq:ode}
	\dot x &= F(x), \\
	x(0) &= x_0
\end{eq}
for a vector field $F \colon \bar\Omega \to \bR^d$.
Within the scope of this work, we do not need the vector field to be defined on the whole Euclidean space as we are only interested in solutions with their trajectories contained in $\bar\Omega$.
In what follows, we make the standing smoothness assumption that
\begin{eq}\label{eq:F_regularity}
	F \in C^1(\bar \Omega)
\end{eq}
throughout. It then holds that $\ddiv F \in L^\infty(\Omega)$ and $\ddiv F \in C(\bar \Omega)$.
The latter implies that $\ddiv F$ and $F$ themselves are uniformly continuous on $\bar \Omega$.

%
%
Let us briefly discuss the existence and uniqueness of solutions of \cref{eq:ode}.
Assumption~\cref{eq:F_regularity} guarantees that $F$ is globally Lipschitz continuous (on $\bar \Omega$).
Hence, there exists a Lipschitz continuous extension to $\bR^d$ according to Kirszbraun's theorem with the same Lipschitz constant, see \cite{bib:Kirszbraun}.
However, this extension does not need to be unique, but it suffices to apply the standard existence and uniqueness theory for dynamical systems.
This guarantees for every such extension the existence of a unique solution $x \colon [0,\infty) \to \bR^d$.

These solutions may, however, depend on the chosen extension, if they ever leave $\bar\Omega$.
To prevent this, we need to ensure trajectories starting in $\bar \Omega$ to remain in $\bar \Omega$ for all $t > 0$.
In other words, we require the set $\bar\Omega$ to be \emph{(positively) invariant} with respect to \cref{eq:ode}.
Solutions with this property are also called \emph{viable}.
There are necessary and sufficient criteria to guarantee the viability of solutions, which can be found in the theorems by Nagumo and Bony--Brezis, see \cite[Theorem 16.5]{bib:AmannMetzen}, \cite[Chapter 4, Section 2, Theorem 2]{bib:AubinCellina_DifferentiaInclusions}, and \cite[Chapter III, §10, XV]{bib:WalterThompson_ODEs}.
Here, we propose the following \emph{no-outflow condition}: For all $x \in \partial\Omega$ we require
%
%
\begin{eq}\label{eq:F_no_outflow}
	F(x) \cdot \nu(x) \leq 0.
\end{eq}

\revb{The reader is reminded that $\nu(x)$ denotes the outward unit normal of the domain's boundary in the point $x \in \partial \Omega$.}{}
\revb{The no-outflow condition \cref{eq:F_no_outflow}}{} resembles the results by Bony, see the aforementioned references as well as \cite{bib:Bony, bib:Redheffer}.
However, the outward unit normal therein has a different definition compared to the one for Lipschitz boundaries, see \cite[A8.5(3)]{bib:AltNuernberg}.

There is another interpretation of \cref{eq:F_no_outflow}.
As it turns out, the KvN equation resembles a transport equation.
For the latter, see for instance \cite{bib:Bardos_Transport}.
Therefore, we need to impose a boundary condition on the so-called \emph{inflow boundary}.
This will be explained in more detail in Section~\ref{sec:existence_kvn}.
Condition \cref{eq:F_no_outflow} will only be occasionally used in the article.
Therefore, it is \emph{not} a standing assumption and will be explicitly stated, where needed.

%
%
Assuming \cref{eq:F_regularity} and \cref{eq:F_no_outflow} hold, the ODE in \cref{eq:ode} has a unique solution on $[0,\infty)$ that takes only values in $\bar\Omega$ for each initial value $x_0 \in \bar\Omega$.
Consider the associated \emph{semiflow}, cf. \cite[Section 10 and Remark (10.2)(c)]{bib:AmannMetzen} and \cite[Theorem 6.1]{bib:Teschl_ODE}, $\Phi \colon [0,\infty) \times \bar\Omega \to \bar\Omega$ with $\Phi_t(x_0) = x(t)$, where $x$ is the solution of \cref{eq:ode}.
Clearly, $\Phi$ has the semigroup properties $\Phi_0 = \id_{\bar\Omega}$ and $\Phi_{t + s} = \Phi_t \circ \Phi_s$.

%
%
\subsubsection{Perron--Frobenius and Koopman Operators} 

Next, we introduce transfer operators.
Our exposition is predominantly based on the corresponding sections in \cite{LaMa94}.
The Perron--Frobenius operator for a given time $t \geq 0$ is the bounded linear operator $\cP(t) \colon L^1(\Omega) \to L^1(\Omega)$ defined by
\begin{eq*}
	\int_A (\cP(t) \rho)(x) \, \dx = \int_{\Phi_t^{-1}(A)} \rho(x) \, \dx \text{~for all~}  A \in \cB(\Omega),
\end{eq*}
where $\cB(\Omega)$ denotes the Borel $\sigma$-algebra on $\Omega$.
Here, the map $\Phi_t$ is nonsingular as it is \reva{a $C^1$-diffeomorphism}{}.
Using the properties of the semiflow $\Phi$, it can be shown that the family $(\cP(t))_{t \geq 0}$ is a semigroup of bounded linear operators on $L^1(\Omega)$, \reva{cf. \cite[p. 42]{LaMa94}}.
This operator is in fact connected to the propagation of probability densities in the Liouville equation.

%
%
In addition, we consider the \emph{Koopman operator} $\cK(t) \colon L^\infty(\Omega) \to L^\infty(\Omega)$ defined by
\begin{eq*}
	\cK(t) f := f \circ \Phi_t.
\end{eq*}
Again, the family $(\cK(t))_{t \geq 0}$ is a semigroup of bounded linear operators on $L^\infty(\Omega)$.
The application of the Koopman operator is connected to the time evolution of an observable along a trajectory.
%
%
It can be shown that the Koopman operator is in fact the dual of the Perron--Frobenius operator with $(L^1(\Omega))^* = L^\infty(\Omega)$, see \reva{\cite[eq. (3.3.4)]{LaMa94}}. Hence, we have
\begin{eq*}
	\langle \cK(t) f, \rho\rangle_{L^\infty(\Omega), L^1(\Omega)}
	= \langle f, \cP(t)\rho\rangle_{L^\infty(\Omega), L^1(\Omega)}
\end{eq*}
for all $f \in L^\infty(\Omega), \rho \in L^1(\Omega)$ and $t \geq 0$.

%
%
The generators of the Perron--Frobenius and Koopman operators are given by
\begin{eq*}
	\cL^* \rho &:= \lim_{t \searrow 0} \frac{1}{t} (\cP(t)\rho - \rho) = -\ddiv(\rho F)&& \text{and}\\
	\cL f &:= \lim_{t \searrow 0} \frac{1}{t} (\cK(t)f - f) = F \cdot \nabla f,&&
\end{eq*}
respectively, for sufficiently smooth functions $f$ and $\rho$.
For now, we do not specify the function spaces and keep it at the formal level.
These operators will be called \emph{Perron--Frobenius generator} and \emph{Koopman generator} to distinguish them from the operators in the associated semigroup.

%
%
\subsubsection{Koopman--von Neumann Operator}

The aim of KvN mechanics is the formulation of a quantum system whose solution $\psi$ satisfies $\rho = |\psi|^2$ according to Born's rule.
In other words, a quantum mechanical system is formulated, whose probability distribution is associated with the solution of the dynamical system in \cref{eq:ode}.
Originally, this framework was proposed as a link between classical and quantum mechanics.
Hence, the infinitesimal generator was first proposed for Hamiltonian systems, see \cite{bib:Koopman, bib:vonNeumann_Koopman, bib:vonNeumann_KoopmanErrata, bib:KoopmanvonNeumann_KvN}.
In recent years, this framework was extended to general ODEs, see \cite{bib:Joseph, bib:BogdanovBognanova, bib:BoganovBogdanovaFastovetsLukichev}.

In this paper, we introduce the KvN generator as
\begin{eq}\label{eq:kvn}
	\kvnop := \frac{1}{2} (\cL^* - \cL).
\end{eq}
This notion diverges slightly from the literature by representing the KvN generator as the skew-symmetric component of the Koopman or Perron-Frobenius generator, respectively.
However, our approach makes the skew-symmetry of the generator immediately evident.
By using the product rule, we formally derive
\begin{eq*}
	\kvnop \psi = -F \cdot \nabla \psi - \frac{1}{2} (\ddiv F) \psi
\end{eq*}
for a sufficiently smooth function $\psi$.
From this we can see that the operator resembles a special case of a \emph{transport equation}.
To the best of the authors' knowledge, this connection has not been highlighted yet.
Throughout this article, we will hence occasionally point out similarities between our approach and the existing literature on transport equations.

\begin{remark}
	In the case of Hamiltonian systems, it is worth mentioning that the Perron--Frobenius and Koopman generators coincide.
	Given a classical Hamiltonian $ H \colon \bR^d \times \bR^d \to \bR $, the equations of motion can be written as
	\begin{equation*}
		\dot{q}_j = \pd{H}{p_j}, \quad \dot{p}_j = -\pd{H}{q_j},
	\end{equation*}
	with $ j = 1, \dots, d $, where the vector $ q $ contains the generalized coordinates and $ p $ the momenta.
	Defining $ x = \begin{bmatrix} q \\[0.2em] p \end{bmatrix} \in \bR^{2 d} $, we obtain $ F = \begin{bmatrix} \phantom{-} \nabla_p H \\ -\nabla_q H \end{bmatrix} $. The Koopman generator can hence be written as
	\begin{equation*}
		\cL f = \sum_{j=1}^d \left( \pd{H}{p_j} \pd{f}{q_j} - \pd{H}{q_j} \pd{f}{p_j} \right)
	\end{equation*}
	and the Perron--Frobenius generator as
	\begin{equation*}
		\cL^* \rho = - \sum_{j=1}^d \left( \pd{H}{p_j} \pd{\rho}{q_j} - \pd{H}{q_j} \pd{\rho}{p_j} \right).
	\end{equation*}
	Compare to \cite[Ex. 7.6.1 and 7.8.1]{LaMa94} for more details.
	That is, $ \cL $ is a skew-adjoint operator with $ \cL = -\cL^* $ and we have $\kvnop = \cL^*$.
\end{remark}

The main task of this paper is to demonstrate the existence of a strongly continuous semigroup of unitary operators on the set of complex $L^2$ functions.
To this end, we wish to apply the results from \cref{ssec:semigroup}, which requires a suitable framework.
This is discussed in the subsequent sections.
In the remainder of this subsection, however, we give an example to demonstrate why this is in fact a nontrivial task.
%
%

One of the necessary conditions for the existence of a strongly continuous semigroup is for the generator to be a densely defined generator with closed graph.
Hence, we need to select an appropriate function space setting.

As we aim at a quantum-physical interpretation of the system, we choose $X = L^2(\Omega)$.
For now, we stick to real-valued functions.
Since the Koopman generator involves a gradient, we might hence try to use the Sobolev space $H_0^1(\Omega)$ as the domain of the KvN generator.
This motivates the following framework.
%
%
Consider $\cL \colon H_0^1(\Omega) \to L^2(\Omega)$ defined by
\begin{eq*}
	\cL\psi := F \cdot \nabla \psi.
\end{eq*}
%
%
In this setting, the Perron--Frobenius generator $\cL^* \colon L^2(\Omega) \to H^{-1}(\Omega)$ as the dual for arbitrary $\varphi \in L^2(\Omega)$ and $\psi \in H_0^1(\Omega)$ reads
\begin{eq*}
	\langle \cL^* \varphi , \psi \rangle_{H^{-1}(\Omega),H_0^1(\Omega)}
	= (\varphi, \cL \psi)_{L^2(\Omega)}
	= (\varphi, F \cdot \nabla \psi)_{L^2(\Omega)}
	= - \langle \ddiv(\varphi F), \psi \rangle_{H^{-1}(\Omega), H_0^1(\Omega)}.
\end{eq*}
Here, $\ddiv$ is the distributional divergence, see~\eqref{eq:distributional_divergence}.
Clearly, if $\psi, \varphi \in H_0^1(\Omega)$, then, using the product rule \cref{eq:product_rule} and the regularity assumption \cref{eq:F_regularity} on $F$, we have
\begin{eq*}
	\ddiv(\varphi F) = F \cdot \nabla \varphi + (\ddiv F) \varphi \in L^2(\Omega).
\end{eq*}%
This means that
\begin{eq}\label{eq:pf_generator_restriction_to_H_0_1}
	\cL^*\varphi = -F \cdot \nabla \varphi - (\ddiv F) \varphi \text{~for~} \varphi \in H_0^1(\Omega)
\end{eq}
and the restriction $\cL^* \colon H_0^1(\Omega) \to L^2(\Omega)$ is a bounded linear operator.
%
%
As we have the continuous embeddings $H_0^1(\Omega) \embeds L^2(\Omega) \embeds H^{-1}(\Omega)$, we can define the KvN generator as $\kvnop \colon H_0^1(\Omega) \to H^{-1}(\Omega)$ by $\kvnop = \frac{1}{2}(\cL^* - \cL)$.
This yields
\begin{eq*}
	\langle \kvnop \varphi, \psi \rangle_{H^{-1}(\Omega),H_0^1(\Omega)}
	&= \frac{1}{2} \left( (\varphi, \cL \psi)_{L^2(\Omega)} - (\cL \varphi, \psi)_{L^2(\Omega)} \right)\\
	&= \frac{1}{2} \left( (F \cdot \nabla \varphi, \psi)_{L^2(\Omega)} - (F \cdot \nabla \psi, \varphi)_{L^2(\Omega)} \right).
\end{eq*}
By construction, this formulation of the KvN generator provides a bounded skew-symmetric operator.

%
%
To apply semigroup theory, however, we need to deal with unbounded operators on $L^2(\Omega)$ taking values in $L^2(\Omega)$ that are defined on a dense subspace.
This motivates in combination with \cref{eq:pf_generator_restriction_to_H_0_1} the definition of the KvN generator as the unbounded operator $\kvnop \colon H_0^1(\Omega) \subseteq L^2(\Omega) \to L^2(\Omega)$ defined by
\begin{eq}\label{eq:kvn_textbook}
	\kvnop \psi = -F \cdot \nabla \psi - \frac{1}{2} (\ddiv F) \psi,
\end{eq}
which is the formula that can be found in the existing literature on KvN mechanics.
Unfortunately, this approach is incompatible to the semigroup framework.
The reason lies in the lack of a closed graph of the KvN generator, even in seemingly trivial cases.

To see this, take the vanishing vector field $F = 0$.
Then, the KvN generator reads $\kvnop = 0$ and its graph is $\gph \kvnop = H_0^1(\Omega) \times \{0\}$.
But as $H_0^1(\Omega) \subseteq L^2(\Omega)$ is dense, the completion of this set is $L^2(\Omega) \times \{0\}$ and therefore the graph cannot be closed.
Thus, a different domain space is required.

%
%
\section{Perron--Frobenius--Sobolev Space}\label{sec:pfs}
As demonstrated in \cref{ssec:derivation_kvn}, the closedness of the KvN generator may suffer from the behavior of the vector field.
In fact, if one would use anything else than a closed subspace of $L^2(\Omega)$, then the graph would not be closed.
Therefore, we want to incorporate the vector field into the definition of the domain.
This motivates the following modification of the Sobolev space.
%
%
\begin{defn}[Perron--Frobenius--Sobolev Spaces]\label{defn:pfs}
	Let $\Omega \subseteq \bR^d$ be a bounded, open domain with Lipschitz boundary.
	The \emph{Perron--Frobenius--Sobolev space} is defined as
	\begin{eq*}
		H(\cL^*, \Omega) := \left\{ \psi \in L^2(\Omega) : \psi F \in H(\ddiv, \Omega)\right\}.
	\end{eq*}
	It is a Hilbert space when equipped with the inner product
	\begin{eq*}
		(\psi , \varphi)_{\pfs} := (\psi , \varphi)_{L^2(\Omega)} + (\ddiv(\psi F) , \ddiv(\varphi F) )_{L^2(\Omega)} \text{~for all~} \psi, \varphi \in \pfs.
	\end{eq*}
\end{defn}
Due to \cref{eq:F_regularity}, the condition in \cref{defn:pfs} is equivalent to $\psi F$ having a weak divergence, i.e., there exists a $y \in L^2(\Omega)$ such that
\begin{eq*}
	(\psi F, \nabla \varphi)_{L^2(\Omega;\bR^d)} = - (y, \varphi)_{L^2(\Omega)} \text{~for all~} \varphi \in C_0^\infty(\Omega).
\end{eq*}
The name of this space is derived from the requirement that the Perron--Frobenius generator is well-defined with values in $L^2(\Omega)$.
In fact, the norm is equivalent to the graph norm of the Perron--Frobenius generator (with constants independent of the space dimensions).
This guarantees the completeness of the space.
We refer to this space as \emph{PFS space} and to its elements as \emph{PFS functions}.

%
%
Transport equations have previously been investigated using semigroup theory, see, e.g., \cite{bib:Bardos_Transport}, where the domain $\{\psi \in L^2(\Omega) \colon F \cdot \nabla \psi \in L^2(\Omega)\}$ was considered.
This is essentially equivalent to $\pfs$ under assumption~\cref{eq:F_regularity}, as the product rule \cref{eq:product_rule} explains.
In this work, we will stick to \cref{defn:pfs} as we inherit more modern tools known from the analysis of $H(\ddiv)$ spaces.

%
%
Next, we discuss a few properties of the PFS space.
By assuming~\cref{eq:F_regularity} and using the product rule, we obtain the inclusions $H^1(\Omega) \subseteq \pfs \subseteq L^2(\Omega)$.
In fact, depending on $F$ and the space dimension, either one of these inclusions can be an equality.
To see this, consider the following two cases:
Firstly, for $d = 1$, let $\Omega \subseteq \bR$ be a finite, open interval and set $F(x) = 1$ for all $x \in \bar\Omega$.
Then, for a function $\psi \in H(\cL^*,\Omega)$, we have $\ddiv(\psi F) = \psi' \in L^2(\Omega)$.
This proves that $H(\cL^*,\Omega) \subseteq H^1(\Omega)$ in this case.
Secondly, for arbitrary dimensions $d \in \bN$, $d \geq 1$, consider the function $F$ with $F(x) = 0$ for all $x \in \Omega$.
Then, we have $\ddiv(\psi F) = 0 \in L^2(\Omega)$ for all $\psi \in L^2(\Omega)$ and hence $L^2(\Omega) = H(\cL^*, \Omega)$.
In conclusion, the Perron--Frobenius--Sobolev space can be interpreted as an intermediate concept between Sobolev and Lebesgue spaces.
%
%
\subsection{Trace Operator and Green's Formula for the Perron--Fro\-be\-nius--Sobolev Functions}\label{ssec:trace_green_pfs}
Trace operators are a vital tool to treat boundary conditions for Sobolev functions.
To declare them on the PFS space, we use the normal trace known for $H(\ddiv)$ functions, see \cref{thm:trace_H_div}.
By definition \reva{we have}{} $\psi F \in H(\ddiv,\Omega)$ for every $\psi \in \pfs$.
Hence, the normal trace of the latter is well-defined with $\tr_\nu(\psi F) \in H^{-1/2}(\partial \Omega)$.
%
%
\begin{thm}[Trace for Perron--Frobenius--Sobolev Functions]\label{thm:trace_pfs}
	Let $\Omega \subseteq \bR^d$ be a bounded, open domain with Lipschitz boundary.
	The map $\tr_{F \nu} \colon \pfs \to H^{-1/2}(\partial \Omega)$ defined by $\psi \mapsto \tr_\nu(\psi F)$ is a bounded linear operator with
	\begin{eq*}
		\|\tr_{F\nu}\|_{\cL\left(\pfs ,H^{-1/2}(\partial \Omega)\right)}
		\leq \max(1, \|F\|_{L^\infty(\Omega;\bR^d)}) \|\tr_\nu\|_{\cL\left(H(\ddiv,\Omega),H^{-1/2}(\partial \Omega)\right)}.
	\end{eq*}
\end{thm}
\begin{proof}
	Take an arbitrary $\psi \in \pfs$.
	Then $\psi F \in H(\ddiv, \Omega)$ is well-defined and we have
	\begin{align*}
		\|\tr_{F\nu}(\psi)\|_{H^{- \scriptstyle\frac{1}{2}}(\Omega)}
		&= \|\tr_{\nu}(\psi F)\|_{H^{- \scriptstyle\frac{1}{2}}(\Omega)}
		\leq \|\tr_\nu\|_{\cL\left(H(\ddiv,\Omega),H^{-1/2}(\partial \Omega)\right)} \|\psi F\|_{H(\ddiv,\Omega)}\\
		&= \|\tr_{\nu}\|_{\cL\left(H(\ddiv,\Omega), H^{-1/2}(\partial\Omega)\right)} \left( \|\psi F\|_{L^2(\Omega;\bR^d)}^2 + \|\ddiv(\psi F)\|_{L^2(\Omega)}^2 \right)^\frac{1}{2}\\
		&\leq \|\tr_{\nu}\|_{\cL\left(H(\ddiv,\Omega), H^{-1/2}(\partial\Omega)\right)} \max(1, \|F \|_{L^\infty(\Omega;\bR^d)}) \|\psi\|_{\pfs}. \qedhere
	\end{align*}
\end{proof}

As mentioned above, traces are relevant when incorporating boundary conditions.
Of particular interest are homogeneous boundary conditions, meaning the trace being zero.
To this end, we define the following space.
%
%
\begin{defn}(Perron--Frobenius--Sobolev functions with vanishing trace)\label{defn:pfz}
	Let $\Omega \subseteq \bR^d$ be a bounded, open domain with Lipschitz boundary.
	The space $H_0(\cL^*,\Omega)$ is defined by
	\begin{eq*}
		H_0(\cL^*, \Omega) := \ker(\tr_{F \nu}) = \{ \psi \in H(\cL^*, \Omega) : \tr_{F \nu}(\psi) = 0\}.
	\end{eq*}
\end{defn}
As this space is the kernel of a linear and bounded operator, it is a closed subspace of $\pfs$ and hence again a Hilbert space with respect to the same norm.

Next Green's formula is reestablished for PFS and Sobolev functions.
%
%
\begin{thm}[Green's Formula, first version]\label{thm:green_pfs_h1}
	Let $\Omega \subseteq \bR^d$ be a bounded, open domain with Lipschitz boundary.
	For all $\psi \in \pfs$ and $\varphi \in H^1(\Omega)$ \reva{we have}{}
	\begin{eq*}
		\langle \tr_{F\nu}(\psi) , \tr(\varphi) \rangle_{H^{-1/2}(\partial \Omega), H^{1/2}(\partial \Omega)}
		&= (\ddiv(\psi F), \varphi)_{L^2(\Omega)} + (\psi , \ddiv(\varphi F))_{L^2(\Omega)}\\
		&\phantom{=}~ - (\ddiv(F) \psi, \varphi)_{L^2(\Omega)}.
	\end{eq*}

\end{thm}
\begin{proof}
	By applying \cref{thm:green_Sobolev_H_div} to $\varphi \in H^1(\Omega)$ and $\psi F \in H(\ddiv,\Omega)$, we obtain
	\begin{eq}\label{eq:tmp_0}
		\langle \tr_{\nu}(\psi F) , \tr(\varphi) \rangle_{H^{-1/2}(\partial \Omega), H^{1/2}(\partial \Omega)}
		= (\ddiv(\psi F), \varphi)_{L^2(\Omega)} + (\psi F, \nabla \varphi)_{L^2(\Omega;\bR^d)}.
	\end{eq}
	With the product rule in \cref{eq:product_rule}, we get $\ddiv(\varphi F) = F \cdot \nabla \varphi + \ddiv(F) \varphi$ and subsequently obtain from  \cref{eq:tmp_0}:
	\begin{eq*}
		(\ddiv(\psi F), \varphi)_{L^2(\Omega)} + (\psi F, \nabla \varphi)_{L^2(\Omega;\bR^d)}
		&= (\ddiv(\psi F), \varphi)_{L^2(\Omega)} + (\psi , \ddiv(\varphi F))_{L^2(\Omega)}\\
		&\phantom{=}~ - (\ddiv(F) \psi, \varphi)_{L^2(\Omega)}.
	\end{eq*}
	Together with $\tr_\nu (\psi) = \tr_{F\nu}(\psi)$, this proves the assertion.
\end{proof}

Due to the symmetry of the left-hand side of the equation in \cref{thm:green_pfs_h1}, one is tempted to interchange the roles of $\psi$ and $\varphi$ therein to extend this result to two functions in $\pfs$.
To do so, the density of $H^1(\Omega)$ in $\pfs$ is required.
Clearly, the Sobolev space is dense in $L^2(\Omega)$ with respect to its norm.
However, it is not clear yet whether the same holds true for the PFS space.
In the following theorem, we address the density of smooth functions in the PFS space.

%
%
\begin{thm}[Density of Smooth Functions]\label{thm:pfs_smooth_density}
	Let $\Omega \subseteq \bR^d$ be a bounded, open domain.
	Then, \reva{the following assertions are true}{}:
	\begin{enum}
		\item\label{enum:pfs_smooth_density} The space $C^\infty(\Omega) \cap \pfs$ is dense in $\pfs$ with respect to the $\pfs$-norm.
		\item\label{enum:pfs_smooth_density_bdry} If, moreover, $\Omega$ has a Lipschitz boundary, then the space $C^\infty(\bar \Omega)$ is dense in $\pfs$ with respect to the $\pfs$-norm.
		\item\label{enum:pfsz_smooth_density} If, moreover, $\Omega$ has a Lipschitz boundary, then the space $C_0^\infty(\Omega)$ is dense in $\pfsz$ with respect to the $\pfs$-norm.
	\end{enum}
\end{thm}
The statement in \cref{thm:pfs_smooth_density}\,\cref{enum:pfs_smooth_density} is a direct counterpart to the Meyers--Serrin theorem, see \cite{bib:MeyersSerrin}.
Due to the technical nature and multiple steps involved in proving~\cref{thm:pfs_smooth_density}, it is included in the appendix.
We will just use it here to propose the following extension of \cref{thm:green_pfs_h1}.
%
%
\begin{thm}[Trace and Green's formula, final version]\label{thm:green_pfs}
	Let $\Omega \subseteq \bR^d$ be a bounded, open domain with Lipschitz boundary.
	Then the following statements hold:
	\begin{enum}
		\item\label{enum:green_extension} The bilinear form $a \colon \pfs \times H^1(\Omega) \to \bR$ defined by $(\psi, \varphi) \mapsto \langle \tr_{F \nu}(\psi) , \tr(\varphi) \rangle$ can be continuously extended to $\pfs \times \pfs$.
		Then we may write $a(\psi, \varphi) = \int_{\partial\Omega} \psi \varphi F \nu \, \dH^{d - 1}$ for all $\psi, \varphi \in \pfs$ and get
		\begin{eq*}
			\int_{\partial\Omega} \psi \varphi F \nu \, \dH^{d - 1} =
			(\ddiv(\psi F), \varphi)_{L^2(\Omega)} + (\psi , \ddiv(\varphi F))_{L^2(\Omega)} - (\ddiv(F) \psi, \varphi)_{L^2(\Omega)}.
		\end{eq*}
		\item\label{enum:green_pfsz} For all $\psi \in \pfsz$ and $\varphi \in \pfs$, we obtain
		\begin{eq*}
			(\ddiv(F) \psi, \varphi)_{L^2(\Omega)} = (\ddiv(\psi F), \varphi)_{L^2(\Omega)} + (\psi, \ddiv(\varphi F))_{L^2(\Omega)}.
		\end{eq*}
		\item\label{enum:trace_extension} If the no-outflow condition \cref{eq:F_no_outflow} is fulfilled, then there exists a bounded linear operator $\tr \colon \pfs \to L^2(\partial \Omega, |F \nu| \, \dH^{d-1})$ with $\tr(\psi) = \psi|_{\partial \Omega}$ for all $\psi \in C^1(\bar \Omega)$, where
		\begin{eq*}
			L^2(\partial \Omega, |F \nu| \, \dH^{d-1}) := \left\{ u \colon \partial \Omega \to \bR : \int_{\partial\Omega} |u|^2 |F \nu| \dH^{d - 1} < \infty \right\}.
		\end{eq*}
	\end{enum}
\end{thm}
\begin{proof}
	%
	%
	ad \cref{enum:green_extension}.
	We define the operator $A \colon H^1(\Omega) \to \pfs^*$ by \begin{eq*}
		\langle A\varphi, \psi \rangle := a(\psi, \varphi) = \langle \tr_{F\nu}(\psi), \tr(\varphi)\rangle.
	\end{eq*}
	By \cref{thm:green_pfs_h1}, we obtain
	\begin{eq*}
		|\langle A \varphi, \psi \rangle|
		&\leq \left|(\psi, \ddiv(\varphi F))_{L^2(\Omega)} + (\ddiv(\psi F), \varphi)_{L^2(\Omega)} - (\ddiv(F)\psi, \varphi)_{L^2(\Omega)}\right|\\
		&\leq \|\psi\|_{L^2(\Omega)} \|\ddiv(\varphi F)\|_{L^2(\Omega)} + \|\ddiv(\psi F)\|_{L^2(\Omega)} \|\varphi\|_{L^2(\Omega)}\\
		&\phantom{\leq}~ + \|\ddiv F\|_{L^\infty(\Omega)} \|\psi\|_{L^2(\Omega)} \|\varphi\|_{L^2(\Omega)}\\
		&\leq (1 + \|\ddiv F\|_{L^\infty(\Omega)}) \|\psi\|_{\pfs} \|\varphi\|_{\pfs}.
	\end{eq*}
	This guarantees that $\|A \varphi\|_{\pfs^*} \leq (1 + \|\ddiv F\|_{L^\infty(\Omega)}) \|\varphi\|_{\pfs}$ for all $\varphi \in H^1(\Omega)$.
	As $C^\infty(\bar \Omega) \subseteq H(\cL^*, \Omega)$ is dense according to \cref{thm:pfs_smooth_density}\,\cref{enum:pfs_smooth_density_bdry}, we obtain that $H^1(\Omega)$ is dense in $\pfs$.
	Hence, there exists a unique, continuous linear extension $A$ on $\pfs$, see \cite[Theorem 3.1--3]{bib:Ciarlet}.

	%
	%
	ad \cref{enum:green_pfsz}.
	Let $\psi \in \pfsz$ and $\varphi \in \pfs$.
	Then we get $\tr_{F \nu}(\psi) = 0$ and hence $\int_{\partial \Omega} \psi \varphi F \nu \, \dH^{d - 1} = 0$.
	The application of \cref{thm:green_pfs}\,\cref{enum:green_extension} yields the assertion.

	%
	%
	ad \cref{enum:trace_extension}.
	We apply \cref{thm:green_pfs}\,\cref{enum:green_extension} with $\varphi = \psi \in C^1(\bar \Omega)$.
	This yields
	\begin{eq*}
		\int_{\partial \Omega} \psi^2 |F \nu| \, \dH^{d - 1}
		&= - \int_{\partial \Omega} \psi^2 F \nu \, \dH^{d - 1} \\
		&=  (\ddiv(F) \psi, \psi)_{L^2(\Omega)} - 2 (\ddiv(\psi F), \psi)_{L^2(\Omega)}\\
		&\leq \|\ddiv(F)\|_{L^\infty(\Omega)} \|\psi\|_{L^2(\Omega)}^2 + 2 \| \ddiv(\psi F) \|_{L^2(\Omega)} \|\psi\|_{L^2(\Omega)}\\
		&\leq \left(\|\ddiv F\|_{L^\infty(\Omega)}^2 + 4 \right)^\frac{1}{2} \|\psi\|_{\pfs} \|\psi\|_{L^2(\Omega)}\\
		&\leq \left(\|\ddiv F\|_{L^\infty(\Omega)}^2 + 4 \right)^\frac{1}{2} \|\psi\|_{\pfs}^2.
	\end{eq*}
	This implies that $\|\psi|_{\partial \Omega}\|_{L^2(\partial \Omega, |F \nu|\dH)} \leq (\|\ddiv(F)\|_{L^\infty(\Omega)}^2 + 4)^\frac{1}{4} \|\psi\|_{\pfs}$ for all $\psi \in C^1(\bar \Omega)$.
	As \cref{thm:pfs_smooth_density}\,\cref{enum:pfs_smooth_density_bdry} guarantees $C^1(\bar \Omega)$ to be dense in $\pfs$, the trace operator can be extended to $\pfs$.
\end{proof}

It should be mentioned that extensions of the trace operator are well-known for the transport equation, see, e.g., \cite[Chapitre 2, Section 4]{bib:Bardos_Transport} and \cite[Proposition 2.1]{bib:DahmenHuangSchwabWelper}.
Here, however, our analysis benefits from the use of the normal trace operator defined for $H(\ddiv)$ functions.
Moreover, to our best knowledge, the density results in \cref{thm:pfs_smooth_density} have not been shown before.

With these results at hand, we close the discussion of the Perron--Frobenius--Sobolev space and draw our attention to the promised existence and uniqueness for the KvN equation.
%
%
\section{Existence and Uniqueness for the Koopman--von Neumann Equation}\label{sec:existence_kvn}
%
%
As semigroup theory is defined for complex Banach spaces, we define the spaces $L^2_\bC(\Omega)$, $\pfsc$, $\pfscz$ to be the sets of complex-valued functions on $\Omega$ such that both the real and imaginary parts are in $L^2(\Omega)$, $H^1(\Omega)$, etc., respectively.
In what follows, we want to apply semigroup theory on $X := L_\bC^2(\Omega)$ equipped with the inner product
\begin{eq*}
	(u , v)_{L^2_\bC(\Omega)} = \int_\Omega \bar u v \, \dx,
\end{eq*}
where $\bar u$ is the (pointwise) conjugate of $u$.
The differential operators $\nabla$ and $\ddiv$ are both separately applied to the real part and imaginary part.
For clarification, we mention that the vector field $F \colon \bar\Omega \to \bR^d$ does \emph{not} take values in the complex numbers.
%
%
\subsection{Koopman--von Neumann Generator}\label{ssec:kvn_generator}

For the definition of the KvN generator, we return to the spaces associated with the real-valued functions.
By construction, the Perron--Frobenius generator $\cL^* \colon \pfs \to L^2(\Omega)$ is well-defined and continuous.
Together with \cref{eq:kvn_textbook} this admits the following definition of the KvN generator.
%
\begin{defn}
	Let $\Omega \subseteq \bR^d$ be a bounded, open domain with Lipschitz boundary.
	The \emph{Koopman--von Neumann generator} $\kvnop \colon \pfs \to \pfs^*$ is defined by
	\begin{eq}\label{eq:kvnop}
		\langle \kvnop \psi, \varphi \rangle_{\pfs^*,\pfs}
		&= \frac{1}{2}\left( (\cL^* \psi, \varphi)_{L^2(\Omega)} - (\psi, \cL^* \varphi)_{L^2(\Omega)} \right)\\
		&= - \frac{1}{2} (\ddiv(\psi F), \varphi)_{L^2(\Omega)} + \frac{1}{2} (\psi, \ddiv(\varphi F))_{L^2(\Omega)}.
	\end{eq}
\end{defn}
Clearly, we have $\langle \kvnop \psi , \varphi \rangle = - \langle \kvnop \varphi, \psi \rangle$.
This guarantees that the generator is skew-symmetric.
From this, we deduce for the dual of the KvN operator that $(\kvnop)^* = - \kvnop$.
Here, we used the reflexivity of the PFS space.

For semigroup theory, we aim to formulate the KvN generator as an operator taking values in $L^2(\Omega)$.
For this sake, we use \cref{thm:green_pfs}.
%
%
\begin{thm}[Koopman--von Neumann generator on Perron--Frobenius--Sobolev spaces]\label{thm:kvnop}
	Let $\Omega \subseteq \bR^d$ be a bounded, open domain with Lipschitz boundary.
	Then the following holds:
	\begin{enum}
		\item\label{enum:kvnop_reformulation} For all $\psi , \varphi \in \pfs$, we have
		\begin{eq*}
			\langle \kvnop \psi, \varphi \rangle_{\pfs^*, \pfs} &= - (\ddiv(\psi F), \varphi)_{L^2(\Omega)} + \frac{1}{2}(\ddiv(F) \psi, \varphi)_{L^2(\Omega)}\\
			&\phantom{=}~ + \frac{1}{2}\int_{\partial \Omega} \psi \varphi F\nu \, \dH^{d - 1}.
		\end{eq*}
		\item\label{enum:kvnop_L2} The restriction $\kvnop \colon \pfsz \to L^2(\Omega)$ is a bounded linear operator with
		\begin{eq*}
			\kvnop \psi = -\ddiv(\psi F) + \frac{1}{2}\ddiv(F) \psi.
		\end{eq*}
	\end{enum}
\end{thm}
\begin{proof}
	%
	%
	ad \cref{enum:kvnop_reformulation}.
	Let $\psi, \varphi \in \pfs$.
	By \cref{thm:green_pfs}\,\cref{enum:green_extension}, we have
	\begin{eq*}
		(\psi, \ddiv(\varphi F))_{L^2(\Omega)} = (\ddiv(F) \psi, \varphi)_{L^2(\Omega)} + \int_{\partial \Omega} \psi \varphi F \nu \, \dH^{d - 1} - (\ddiv(\psi F), \varphi)_{L^2(\Omega)}.
	\end{eq*}
	Plugging this into the definition of the KvN generator in \cref{eq:kvnop} yields the desired result.

	%
	%
	ad \cref{enum:kvnop_L2}.
	The linearity of the KvN generator is clear by definition. If $\psi \in \pfsz$, we have $\int_{\partial \Omega} \psi \varphi \, \dH^{d - 1} = 0$ for all $\varphi \in \pfs$.
	This yields
	$\kvnop \psi = -\ddiv(\psi F) + \frac{1}{2}\ddiv(F) \psi \in L^2(\Omega)$.
	Moreover, the estimate
	\begin{eq*}
		\|\kvnop \psi\|_{L^2(\Omega)}
		&\leq \|\ddiv(\psi F)\|_{L^2(\Omega)} + \frac{1}{2}\|\ddiv(F)\|_{L^\infty(\Omega)} \|\psi\|_{L^2(\Omega)}\\
		&\leq \sqrt{1 + \frac{1}{4} \|\ddiv(F)\|_{L^\infty(\Omega)}^2} \|\psi\|_{\pfs}
	\end{eq*}
	shows the boundedness of the KvN generator.
\end{proof}
Here, we make two important observations:
On the one hand, the KvN generator on $\pfs$ has a part involving the trace.
Hence, it cannot take values in $L^2(\Omega)$ if the trace of $\psi$ does not vanish.
It is thus necessary to restrict the generator to the subspace $\pfsz$.
On the other hand, the KvN equation takes the form of a \emph{transport equation}.
In the latter context, the boundary is usually divided into three parts:
$\Gamma_\mp := \{ x \in \partial\Omega : F \cdot \nu \lessgtr 0\}$ and $\Gamma_0 := \{x \in \partial\Omega : F \cdot \nu = 0 \}$.
The first two sets are referred to as \emph{inflow} and \emph{outflow boundary}, whereas the third is sometimes called \emph{characteristic boundary}, see \cite{bib:DahmenHuangSchwabWelper,bib:BroersenDahmenStevenson} or \emph{solid wall}, see \cite[Subsection 2.2.1]{bib:KuzminHaemaelaeinen}.
Existence theory for this class of equations has been derived in, e.g., \cite{bib:Bardos_Transport}.
It is necessary to provide a condition on the inflow boundary, but it is in general \emph{not} possible to enforce one on the other parts.
Here, we need to force the traces of $\psi F \cdot \nu$ to take the value zero on the entire boundary for the reasons outlined above.
This is only feasible under the condition in \cref{eq:F_no_outflow}, which is already relevant to guarantee that the trajectories of \cref{eq:ode} are contained in $\bar\Omega$, cf.~\cref{ssec:derivation_kvn}.
In a physical sense, this setting can be interpreted as a particle to be confined in the closed domain.
%
%
\subsection{Koopman--von Neumann Equation}\label{ssec:existence_kvn}
Now, we turn to the evolution equation \cref{eq:evo} driven by the KvN operator.
First, we define the KvN generator for complex-valued functions $\kvnop \colon \pfscz \to L^2_\bC(\Omega)$ by
\begin{eq*}
	\kvnop \psi := \kvnop(\re \psi) + \ic \kvnop(\im \psi).
\end{eq*}
Here, we use the same notation for the real and complex case.
Then, the inner product with $\varphi \in L^2_\bC(\Omega)$ reads
\begin{eq*}
	(\kvnop \psi , \varphi)_{L^2_\bC(\Omega)} = - \int_\Omega \ddiv(\bar \psi F) \varphi \, \mathrm{d}x + \frac{1}{2} \int_\Omega \ddiv(F) \bar \psi \varphi \, \mathrm{d}x.
\end{eq*}
Based on the discussion in the previous subsection, we formulate the following evolution equation
\begin{eq}\label{eq:kvn_evo}
	\partial_t \psi &= \kvnop \psi, &&\text{for all~} t \geq 0 \text{~in~} \Omega ,\\
	\psi F \cdot \nu      &= 0,           &&\text{for all~} t \geq 0 \text{~on~} \partial \Omega ,\\
	\psi(0)         &= \psi_0,      &&\text{in~} \Omega.
\end{eq}
Our goal is to use \cref{cor:lumer_phillips_cor}.
In what follows, we verify its assumptions.
%
%
\begin{lem}[Densely defined, closed graph]\label{lem:kvnop_domain_gph}
	The KvN generator is a densely defined, closed, linear operator $\kvnop \colon \cD(\kvnop) \subseteq L^2_\bC(\Omega) \to L^2_\bC(\Omega)$ with $\cD(\kvnop) := \pfszc$.
\end{lem}
\begin{proof}
	The proof is only shown for the real-valued case $\kvnop \colon \cD(\kvnop) \subseteq L^2(\Omega) \to L^2(\Omega)$ with $\cD(\kvnop) = \pfsz$.
	The complex case is derived by using the following arguments on real and imaginary parts, respectively.

	As $C^\infty_0(\Omega)$ is dense in $L^2(\Omega)$, we obtain the density of $\pfsz$ in $L^2(\Omega)$.
	Thus, it is left to show the closedness of $\gph{\kvnop}$ in $L^2(\Omega) \times L^2(\Omega)$.
	Let a sequence $(\psi_n, y_n)_{n \in \bN} \subseteq \gph{\kvnop}$ be given that converges in $L^2(\Omega) \times L^2(\Omega)$ to $(\psi, y)$.
	We have to show that $(\psi, y) \in \gph{\kvnop}$.
	To show $\psi \in \pfs$ we use \cref{defn:pfs} and take an arbitrary $\varphi \in C_0^\infty(\Omega)$.
	For all $n \in \bN$, we have
	\begin{eq*}
		(\psi_n F, \nabla \varphi)_{L^2(\Omega;\bR^d)}
		&= - (\ddiv(\psi_n F), \varphi)_{L^2(\Omega)}
		= (\kvnop \psi_n, \varphi)_{L^2(\Omega)} - \frac{1}{2}(\ddiv(F) \psi_n, \varphi)_{L^2(\Omega)}\\
		&= (y_n, \varphi)_{L^2(\Omega)} - \frac{1}{2}(\ddiv(F) \psi_n, \varphi)_{L^2(\Omega)}.
	\end{eq*}
	The convergence of $\psi_n \to \psi$ and $y_n \to y$ in $L^2(\Omega)$ yields
	\begin{eq*}
		(\psi F, \nabla \varphi)_{L^2(\Omega;\bR^d)}
		&= \lim_{n \to \infty} (\psi_n F, \nabla \varphi)_{L^2(\Omega;\bR^d)}
		= \lim_{n \to \infty} \left( (y_n, \varphi)_{L^2(\Omega)} - \frac{1}{2} (\ddiv(F) \psi_n , \varphi)_{L^2(\Omega)} \right)\\
		&= (y, \varphi)_{L^2(\Omega)} - \frac{1}{2} (\ddiv(F) \psi, \varphi)_{L^2(\Omega)}.
	\end{eq*}
	This implies that $y \in \pfs$ with $-\ddiv(\psi F) = y - \frac{1}{2}\ddiv(F) \psi$.
	Moreover, we estimate
	\begin{eq*}
		\|\ddiv(\psi_n F) - \ddiv(\psi F)\|_{L^2(\Omega)}
		&= \left\| \frac{1}{2} \ddiv(F) ( \psi_n - \psi) - (y_n - y) \right\|_{L^2(\Omega)}\\
		&\leq \frac{1}{2} \|\ddiv(F)\|_{L^\infty(\Omega)} \|\psi_n - \psi\|_{L^2(\Omega)} + \|y_n - y\|_{L^2(\Omega)} \to 0
	\end{eq*}
	as $n \to \infty$.
	Together with $\psi_n \to \psi$ in $L^2(\Omega)$ this yields $\psi_n \to \psi$ in $\pfs$.
	Since $\pfsz \subseteq \pfs$ is a closed subspace and $\psi_n \to \psi$ in $\pfsz$, we obtain $\psi \in \pfsz$.
	Moreover, we get
	\begin{eq*}
		\kvnop \psi
		= -\ddiv(\psi F) + \frac{1}{2} \ddiv(F) \psi
		= y.
	\end{eq*}
	Hence, we deduce $(\psi, y) \in \gph{\kvnop}$.
\end{proof}

Next, we show that the KvN operator is dissipative.
%
%
\begin{thm}[Dissipativity]\label{thm:kvn_dissipative}
	Let $\Omega \subseteq \bR^d$ be a bounded, open domain with Lipschitz boundary.
	The KvN operator $\kvnop: \pfszc \subseteq L^2_\bC(\Omega) \to L^2_\bC(\Omega)$ is dissipative.
	\reva{Even more}{}, we have $\re\left( (\kvnop \psi, \psi)_{L^2_\bC(\Omega)} \right) = 0$.
\end{thm}
\begin{proof}
	Let $\psi \in \pfscz$.
	Using the aforementioned skew-symmetry, we find
	\begin{eq*}
		(\kvnop \psi, \psi)_{L^2_\bC(\Omega)}
		&= \int_\Omega \overline{(\kvnop \psi)} \psi \, \dx\\
		&= \int_\Omega \left(\kvnop(\re(\psi)) - \ic \kvnop(\im(\psi))\right)(\re(\psi) + \ic \im(\psi)) \, \dx\\
		&= (\kvnop(\re(\psi)), \re(\psi))_{L^2(\Omega)} + (\kvnop(\im(\psi)), \im(\psi))_{L^2(\Omega)}\\
		&\phantom{=}~ + \ic (\kvnop(\re(\psi)), \im(\psi))_{L^2(\Omega)}
		- \ic (\kvnop(\im(\psi)), \re(\psi))_{L^2(\Omega)}\\
		&= 2 \ic (\kvnop(\re(\psi)), \im(\psi))_{L^2(\Omega)}.
	\end{eq*}
	This implies that $\re\left((\kvnop \psi, \psi)_{L^2_\bC(\Omega)}\right) = 0$, which completes the proof.
\end{proof}

With these results at hand, we are ready to prove the main result of this work.
%
%
\begin{thm}[Existence and uniqueness of solutions of \cref{eq:kvn_evo}]
	Let $\Omega \subseteq \bR^d$ be a bounded, open domain with Lipschitz boundary \reva{and assume the no-outflow condition \cref{eq:F_no_outflow} to hold.
	Moreover, let an initial value $\psi_0 \in \pfszc$ be given.}{}
	Then the KvN generator induces a $C_0$-semigroup of contractions $(T(t))_{t \geq 0}$ and \cref{eq:kvn_evo} has a unique solution $\psi \in C^1([0,\infty), L^2_\bC(\Omega)) \cap C([0,\infty), \pfszc)$ defined by $\psi(t) = T(t)\psi_0$.
	Moreover, for all $t \geq 0$ \reva{we have}{} $\|\psi(t)\|_{L^2_\bC(\Omega)} = \|\psi_0\|_{L^2_\bC(\Omega)}$.
\end{thm}
\begin{proof}
	We verify the conditions of \cref{cor:lumer_phillips_cor}.
	As shown in \cref{lem:kvnop_domain_gph}, the KvN operator is densely defined in $L^2_\bC(\Omega)$ and has a closed graph in $L^2_\bC(\Omega) \times L^2_\bC(\Omega)$.
	From \cref{thm:kvn_dissipative}, we get that $\kvnop$ is dissipative.
	For $\psi, \varphi \in \pfszc$ we have
	\begin{eq*}
		(\kvnop \psi, \varphi)_{L^2_\bC(\Omega)}
		&= (\kvnop \re\psi, \re\varphi)_{L^2(\Omega)}
		+ (\kvnop \im\psi, \im\varphi)_{L^2(\Omega)}\\
		&\phantom{=}~ + \ic (\kvnop \re\psi, \im \varphi)_{L^2(\Omega)}
		- \ic (\kvnop\im\psi, \re\varphi)_{L^2(\Omega)}\\
		&= - (\kvnop \re\varphi, \re\psi)_{L^2(\Omega)}
		- (\kvnop \im\varphi, \im\psi)_{L^2(\Omega)}\\
		&\phantom{=}~ - \ic (\kvnop \im\varphi, \re\psi)_{L^2(\Omega)}
		+ \ic (\kvnop \re\varphi \im\psi)_{L^2(\Omega)} \\
		&= -(\kvnop\varphi, \psi)_{L^2_\bC(\Omega)}.
	\end{eq*}
	This shows also for the complex case that $(\kvnop)^* = -\kvnop$.
	By \cref{thm:kvn_dissipative}, we obtain that $\kvnop$ and its dual are both dissipative.
	Then, \cref{cor:lumer_phillips_cor} guarantees the existence of a contractive $C_0$-semigroup $(T(t))_{t \in [0,\infty)}$.
	By \cref{thm:evo_existence_uniqueness}, the function $\psi \colon [0,\infty) \to \pfscz$ and $t \mapsto T(t)\psi_0$ is the unique solution of \cref{eq:kvn_evo} and satisfies $\psi \in C^1([0,\infty), L^2_\bC(\Omega)) \cap C([0,\infty), \pfszc)$.
	Moreover, we deduce
	\begin{eq*}
		\frac{d}{dt} \|\psi(\cdot)\|_{L^2_\bC(\Omega)}^2 (t)
		&= (\partial_t \psi(t), \psi(t))_{L^2_\bC(\Omega)} + (\psi(t), \partial_t \psi(t))_{L^2_\bC(\Omega)}\\
		&= (\kvnop \psi (t), \psi(t))_{L^2_\bC(\Omega)} + (\psi(t), \kvnop\psi(t))_{L^2_\bC(\Omega)}\\
		&= 2 \re\left( (\kvnop \psi(t), \psi(t))_{L^2_\bC(\Omega)} \right) \\
		&= 0
	\end{eq*}
	and hence $\|\psi(t)\|_{L_\bC^2(\Omega)}^2 = \|\psi_0\|_{L^2_\bC(\Omega)}^2$ for all $t \geq 0$.
\end{proof}

%
%
\section{Conclusion and Outlook}\label{sec:conclusion}

In this article, we have proven the existence and uniqueness of solutions of the KvN equation associated with an autonomous initial value problem on a bounded, open domain with Lipschitz boundary.
To this end, we have introduced and analyzed an extension of Sobolev spaces and derived its properties.
The analysis turned out to be closely related to the one for the transport equation.

For nonautonomous ODEs, a transition to time-dependent KvN generators is required.
We are confident that this goal can be achieved for sufficiently regular vector fields.
Particularly, the density of smooth functions pointed out in \cref{thm:pfs_smooth_density} will very likely be useful for this purpose.
The latter might also be a valuable addition to the theory of transport equations.

As there is a great variety of dynamical systems, we must expect a great variety in the properties of the KvN semigroup, too.
For this reason, its spectral properties as well as the ones for the generator are of great interest.
\begin{envrevb}
	The connection with transport equations opens new avenues for the numerical solution of the KvN equation and its implementation on a quantum computer.
	For this we give a brief sketch:
	First, we discretize the KvN equation in space using methods such as finite differences, finite elements, or finite volumes.
	This results in a discretized Koopman--von Neumann operator represented as a matrix.
	Since the associated matrix exponential cannot generally be computed analytically, we employ approximation methods like the Trotter product formula.
	This approach may allow to decompose the operator into one- and two-qubit operations suitable for execution on a quantum computer.
	These aspects, however, will be analyzed in future work.
\end{envrevb}
%
%
\section*{Acknowledgments}
The authors thank Sebastian Knebel and Arwed Steuer for their constructive feedback on an earlier version of this manuscript. 

\section*{Funding and Competing Interests}
This study was funded within the \emph{Einstein Research Unit Perspectives of a quantum digital transformation: Near-term quantum computational devices and quantum processors} and the \emph{QuantERA II Programme} that has received funding from the European Union's Horizon 2020 research and innovation programme under Grant Agreement No. 101017733.
The authors have no competing interests to declare that are relevant to the content of this article.
%
%
\printbibliography

%
%
\appendix%
\counterwithin{thm}{section}%
\section{Appendix --- Proof of Theorem \ref{thm:pfs_smooth_density}}\label{apx:density}

This appendix is devoted to the proofs of the density statements regarding smooth functions presented in \cref{thm:pfs_smooth_density}.
First, we collect some additional tools from functional analysis.
We refer to \cite[Sections 4.13--4.15, 4.20, 4.21, 4.23, 4.24]{bib:AltNuernberg} and \cite[Section 4.2]{bib:EvansGariepy}.

%
%
Let $u \in L^1(\bR^d)$ and $G \colon \bR^d \times \bR^d \to \bR$ be Lebesgue measurable.
The \emph{convolution} {${G * u} \colon \bR^d \to \bR$} is defined by
\begin{eq*}
	(G * u)(x) := \int_{\bR^d} G(x,y) u(x - y) \,\mathrm{d}y.
\end{eq*}
We have the estimate
\begin{eq}\label{eq:convolution_estimate}
	\|G * u\|_{L^p(\bR^d)} \leq \|u\|_{L^1(\bR^d)} \sup_{h \in \supp(u)} \|G(\cdot + h, \cdot)\|_{L^p(\bR^d)},
\end{eq}
if the supremum on the right-hand side exists and is finite.
Let $g \colon \bR^d \to \bR$ only depend on one variable, then the convolution reads
\begin{eq*}
	(g * u)(x) = \int_{\bR^d} g(y) u(x - y) \,\mathrm{d}y.
\end{eq*}
By a change of variables, one can show that $g * u = u * g$.
In this case, for $g \in L^p(\bR^d)$ and $p \in [1,\infty)$, we get the estimate
\begin{eq}\label{eq:convolution_estimate_simple}
	\|u * g\|_{L^p(\bR^d)} \leq \|u\|_{L^1(\bR^d)} \|g\|_{L^p(\bR^d)}.
\end{eq}

%
%
Let $\rho \in C^\infty(\bR^d)$ with $\rho \geq 0$ on $\bR^d$ and $\supp \rho \subseteq \bB$ be given, where $\bB$ is the closed unit ball in $\bR^d$.
Further, we assume $\int_{\bR^d} \rho \, \dx = 1$.
Consider the \emph{standard Dirac sequence} associated with $\rho$, which is the sequence $(\rho_\eps)_{\eps > 0}$ of functions $\rho_\eps \colon \bR^d \to \bR$ defined by $\rho_\eps(x) := \eps^{-d} \rho\left(\frac{x}{\eps}\right)$.
The following properties are satisfied:
\begin{enum}
	\item $\rho_\eps \geq 0$ on $\bR^d$,
	\item $\supp \rho_\eps \subseteq \eps \bB$, and
	\item $\int_{\bR^d} \rho_\eps \, \dx = 1$.
\end{enum}
From now on, let $(\rho_\eps)_{\eps > 0}$ be a standard Dirac sequence.
%
%
One can show that the convolution of a function $u \in L^p(\Omega)$ for $p \in [1,\infty)$ with $\rho_\eps$ is again smooth, see \cite[Section 4.13 (4)]{bib:AltNuernberg}.
In this case, $\rho_\eps$ is also called \emph{mollifier} and the convolution process is called \emph{mollification} and we obtain the following property for the derivatives:
\begin{eq}\label{eq:mollifier:derivative}
	\partial_j(u * \rho) = u * \partial_j \rho.
\end{eq}
If $u$ has in fact more regularity, one can interchange the roles in \cref{eq:mollifier:derivative}.

%
%
Let $\Omega \subseteq \bR^d$ be an open, bounded domain.
A \emph{locally finite open covering} of $\Omega$ is a family of nonempty, open sets $(U_k)_{k \in \bN}$ in $\bR^d$ such that $\Omega \subseteq \bigcup_{k \in \bN} U_k$ and $U_k \cap \Omega \neq \emptyset$ and for all $x \in \bigcup_{k \in \bN} U_k$ exists $\eps > 0$ such that $\{k \in \bN : U_k \cap (x + \eps \bB)\}$ is finite.
%
%
A frequently used example (cf. \cite[Section 4.21, (3)]{bib:AltNuernberg}) is defined by
\begin{eq}\label{eq:open_covering_example}
	U_k = \left\{x \in \Omega : \diam{\Omega}\frac{1}{2^{k + 1}} < \dist(x, \partial \Omega) < \diam{\Omega}\frac{1}{2^{k - 1}}\right\},
\end{eq}
where $\diam{\Omega} = \sup\{ \|x - x'\| \colon  x, x' \in \Omega\}$ is the diameter (of $\Omega$) and $\dist(\cdot, \partial\Omega) \colon \bR^d \to \bR$ denotes the distance function (with respect to $\partial \Omega$).

%
%
A \emph{partition of unity} associated with a locally finite open covering is a family of smooth functions $(\eta_k)_{k \in \bN}$ with $\eta_k \in C_0^\infty(U_k)$, $\eta_k \geq 0$ and
\begin{eq*}
	\sum_{k \in \bN} \eta_k = 1 \text{~on~} \Omega.
\end{eq*}
In this sum, the functions $(\eta_k)_{k \in \bN}$ are smoothly extended by zero and for every argument only finitely many of them have nonzero values.
With these additional tools, we are ready to prove \cref{thm:pfs_smooth_density}.

%
%
For a function depending on two variables $x, y$, we denote by $\nabla_x, \ddiv_x$ and $\nabla_y, \ddiv_y$ the gradients and divergences with respect to $x$ and $y$, respectively.
We start by deriving the following auxiliary result.
%
%
\begin{lem}\label{lem:comparison}
	Let $\Omega \subseteq \bR^d$ be a bounded, open domain with Lipschitz boundary.
	Consider for fixed $\delta > 0$ a set $D_\delta := \{x \in \Omega: \dist(x,\partial\Omega) > \delta\}$.
	For $\eps < \delta$ and $u \in L^2(\Omega)$ define $u_\eps := \rho_\eps * u$ and $(uF)_\eps = \rho_\eps * (u F)$.
	Then, we obtain
	\begin{eq*}
		\|\ddiv(u_\eps F - (u F)_\eps )\|_{L^2(D_\delta)} \to 0 \text{~as~} \eps \to 0.
	\end{eq*}
\end{lem}
\begin{proof}
	The proof is split into several steps.
	%
	%
	\paragraph{Step 1.} There exists a constant $C > 0$ that depends only on $F$ and $\rho$ such that \linebreak \mbox{$\|\ddiv(u_\eps F - (u F)_\eps )\|_{L^2(D_\delta)} \leq C \|u\|_{L^2(D_\delta)}$}.\\

	\noindent For $x \in D_\delta$, we get
	\begin{eq*}
		\ddiv_x(u_\eps F &- (u F)_\eps)(x)
		= \ddiv_x\left( \int_\Omega \rho_\eps(x - y) u(y) F(x) \, \mathrm{d}y \right) - \ddiv_x\left( \int_\Omega \rho_\eps(x - y) u(y) F(y) \, \mathrm{d}y \right)\\
		&= \int_\Omega \nabla_x \rho_\eps(x - y)\cdot u(y) F(x) \, \mathrm{d}y + \int_\Omega \rho_\eps(x - y) u(y) \ddiv F(x) \, \mathrm{d}y\\
		&\phantom{=}~ - \int_\Omega \nabla_x \rho_\eps(x - y) \cdot u(y) F(y) \, \mathrm{d}y\\
		&= \sum_{j = 1}^d \int_\Omega \partial_j \rho_\eps(x - y) u(y)(F_j(x) - F_j(y)) \, \mathrm{d}y + \int_\Omega \rho_\eps(x - y) u(y) \ddiv F(x) \, \mathrm{d}y\\
		&= \sum_{j = 1}^d (\partial_j \rho_\eps * G_j) (x) + \ddiv F(x) (\rho_\eps * u)(x),
	\end{eq*}
	where $G_j(x,y) := u(y)(F_j(x) - F_j(y))$.
	By extending $F$ to a Lipschitz continuous function on $\bR^d$ by Kirszbraun's theorem and the other $L^2$-functions by zero outside $\Omega$, we get the following estimates by using \cref{eq:convolution_estimate_simple}.
	First, we obtain
	\begin{eq*}
		\| (\rho_\eps * u) \ddiv F\|_{L^2(D_\delta)}
		&\leq \|\ddiv F\|_{L^\infty(D_\delta)} \|\rho_\eps * u\|_{L^2(D_\delta)}
		\leq \|\ddiv F\|_{L^\infty(\Omega)} \|\rho_\eps * u\|_{L^2(\bR^d)}\\
		&\leq \|\ddiv F\|_{L^\infty(\Omega)} \|\rho_\eps\|_{L^1(\bR^d)} \|u\|_{L^2(\bR^d)}
		= \|\ddiv F\|_{L^\infty(\Omega)} \|u\|_{L^2(\Omega)}.
	\end{eq*}
	Further, we get for all $j = 1, \dots, d$ by \cref{eq:convolution_estimate} that
	\begin{eq*}
		\|\partial_j \rho_\eps * G_j\|_{L^2(D_\delta)}
		\leq \|\partial_j \rho_\eps * G_j\|_{L^2(\bR^d)}
		\leq \|\partial _j \rho_\eps\|_{L^1(\bR^d)} \sup_{h \in \supp(\rho_\eps)} \|G_j(\cdot + h, \cdot)\|_{L^2(\bR^d)}.
	\end{eq*}
	Regarding the first factor, we derive
	\begin{eq*}
		\int_{\bR^d} |\partial_j \rho_\eps(x)| \dx
		= \int_{\bR^d} \frac{1}{\eps^{d + 1}} \left|\partial_j \rho\left(\frac{x}{\eps}\right)\right| \dx
		= \frac{1}{\eps} \int_{\bR^d} |\partial_j \rho(x)| \dx
		= \frac{1}{\eps} \|\partial_j \rho\|_{L^1(\bR^d)}.
	\end{eq*}
	For the second part, we get $\supp(\rho_\eps) = \eps \bB$ and for $h \in \eps \bB$ we estimate
	\begin{eq*}
		\|G_j(\cdot + h, \cdot)\|_{L^2(\bR^d)}^2
		&= \int_{\bR^d} |G_j(x + h, x)|^2 \dx
		= \int_{\bR^d} u(x)^2 (F_j(x + h) - F_j(x))^2 \dx\\
		&\leq \Lip(F_j)^2 |h|^2 \int_\Omega u(x)^2 \dx
		= \Lip(F_j)^2 \eps^2 \|u\|_{L^2(\Omega)}^2.
	\end{eq*}
	This yields
	\begin{eq*}
		\|\partial_j \rho_\eps * G_j\|_{L^2(D_\delta)}
		\leq \Lip(F) \|\partial_j \rho\|_{L^1(\bR^d)} \|u\|_{L^2(\Omega)}.
	\end{eq*}
	In combination, we obtain
	\begin{eq*}
		\|\ddiv(u_\eps F - (uF)_\eps)\|_{L^2(D_\delta)}
		\leq \left( \| \ddiv F \|_{L^\infty(\Omega)} + \sum_{j = 1}^d \Lip(F) \|\partial_j \rho\|_{L^1(\bR^d)} \right) \|u\|_{L^2(\Omega)}.
	\end{eq*}
	This proves the assertion of Step 1.
	%
	%
	\paragraph{Step 2.} If $u \in C^1(\bar\Omega)$, then $\|\ddiv(u_\eps F - (u F)_\eps)\|_{L^2(D_\delta)} \to 0$.\\

	\noindent By \cref{eq:mollifier:derivative}, we get
	\begin{eq*}
		&\ddiv(u_\eps F - (u F)_\eps)(x)
		= \nabla u_\eps(x) \cdot F(x) + u_\eps(x) \ddiv F(x) - (\rho_\eps * \ddiv(u F))(x)\\
		&= \int_\Omega \rho_\eps(x - y) \nabla_y u(y) \cdot F(x) \, \mathrm{d}y + \int_\Omega \rho_\eps(x - y) u(y) \ddiv F(x) \, \mathrm{d}y - \int_\Omega \rho_\eps(x - y) \ddiv(u F)(y) \, \mathrm{d}y\\
		&= \int_\Omega \rho_\eps(x - y) \nabla_y u(y) \cdot (F(x) - F(y)) \, \mathrm{d}y + \int_\Omega \rho_\eps(x - y)u(y) (\ddiv F(x) - \ddiv F(y)) \, \mathrm{d}y\\
		&= \rho_\eps * G'(x) + \rho * G''(x),
	\end{eq*}
	where $G'(x,y) := \nabla u(y) \cdot (F(x) - F(y))$ and $G''(x,y) := u(y)(\ddiv F(x) - \ddiv F(y))$.
	Then, we obtain
	\begin{eq*}
		\|\rho_\eps * G'\|_{L^2(D_\delta)}
		\leq \|\rho_\eps\|_{L^1(\bR^d)} \sup_{h \in \supp(\rho_\eps)}\|G'(\cdot + h, \cdot)\|_{L^2(\bR^d)}.
	\end{eq*}
	Further, with $\supp(\rho_\eps) = \eps \bB$ and $h \in \eps \bB$, we get
	\begin{eq*}
		\|G'(\cdot + h, \cdot)\|_{L^2(\bR^d)}^2
		&= \int_{\bR^d} G'(x + h,x)^2 \dx
		\leq \int_{\bR^d} (\nabla u(x) \cdot (F(x + h) - F(x))^2 \dx\\
		&\leq \Lip(F)^2 \eps^2 \int_{\bR^d} |\nabla u(x)|^2 \dx.
	\end{eq*}
	Similarly, as $\ddiv F$ is uniformly continuous on $\bar\Omega$, we find for every $\eps' > 0$ an $\eps_0 > 0$ such that
	\begin{eq*}
		|\ddiv F(x) - \ddiv F(y)| \leq \eps' \text{~for all~} x,y \in \bar\Omega \text{~with~} |x - y| < \eps_0.
	\end{eq*}
	Then, we get for $\eps \leq \eps_0$ that $|h| < \eps_0$ and thus
	\begin{eq*}
		\|\rho_\eps * G''\|_{L^2(D_\delta)}
		&\leq \|\rho_\eps\|_{L^1(\bR^d)} \sup_{h \in \supp(\rho_\eps)} \|G(\cdot + h, \cdot)\|_{L^2(\bR^d)}\\
		&\leq \sup_{h \in \eps \bB} \left( \int_{\bR^d} u(x)^2 (\ddiv F(x + h) - \ddiv F(x))^2 \dx \right)^\frac{1}{2}
		\leq \eps' \|u\|_{L^2(\Omega)}.
	\end{eq*}
	Hence, for arbitrary $\eps' > 0$ we take $\eps := \min(\eps_0,\frac{1}{1 + \Lip(F)} \eps')$.
	Then, we get
	\begin{eq*}
		\|\ddiv(u_\eps F) - \ddiv((uF)_\eps)\|_{L^2(D_\delta)}
		&\leq \|\rho_\eps * G'\|_{L^2(D_\delta)} + \|\rho_\eps * G''\|_{L^2(D_\delta)}\\
		&\leq \Lip(F)\eps \|\nabla u\|_{L^2(\Omega;\bR^d)} + \eps' \|u\|_{L^2(\Omega)}\\
		&\leq \eps' \|u\|_{H^1(\Omega)}.
	\end{eq*}
	This shows that $\|\ddiv(u_\eps F - (u F)_\eps)\|_{L^2(D_\delta)} \to 0$ as $\eps \searrow 0$.
	%
	%
	\paragraph{Step 3.} $\| \ddiv(u_\eps F - (u F)_\eps) \|_{L^2(D_\delta)} \to 0$ as $\eps \to 0$ for all $u \in L^2(\Omega)$.\\

	\noindent As a first step, we use the density $C^\infty(\Omega) \subseteq L^2(\Omega)$, see \cite{bib:AltNuernberg}.
	Hence, for every $u \in L^2(\Omega)$ and every $\alpha > 0$, there exists a function $u^{(\alpha)} \in C^1(\bar \Omega)$ with $\|u - u^{(\alpha)}\|_{L^2(\Omega)} \leq \alpha$.
	Using the results from Step 1 and Step 2 we estimate that
	\begin{eq*}
		\|\ddiv\left(u_\eps F - (u F)_\eps\right)\|_{L^2(D_\delta)}
		&\leq \left\|\ddiv\left(\left(u - u^{(\alpha)}\right)_\eps F - \left(\left(u - u^{(\alpha)}\right)F\right)_\eps\right)\right\|_{L^2(D_\delta)}\\
		&\phantom{\leq}~ + \|\ddiv((u^{(\alpha)})_\eps F - (u^{(\alpha)} F)_\eps)\|_{L^2(D_\delta)}\\
		&\leq C \|u - u^{(\alpha)}\|_{L^2(\Omega)} + \|\ddiv((u^{(\alpha)})_\eps F - (u^{(\alpha)} F)_\eps)\|_{L^2(D_\delta)}.
	\end{eq*}
	Hence, if we let $\eps \searrow 0$, we get by Step 2 that
	\begin{eq*}
		\limsup_{\eps \to 0} \|\ddiv(u_\eps F - (u F)_\eps)\|_{L^2(D_\delta)} \leq C \|u - u^{(\alpha)}\|_{L^2(\Omega)} \leq C \alpha.
	\end{eq*}
	As $\alpha > 0$ was chosen arbitrarily, we obtain the assertion.
\end{proof}

\cref{lem:comparison} shows that the difference of divergences between the product of the mollification with the vector field and the mollification of the product vanishes asymptotically.
Next, the approximation on the subdomain $D_\delta$ is discussed.
%
%
\begin{lem}\label{lem:aux_apx}
	Let $\psi \in \pfs$.
	Consider for fixed $\delta > 0$ a set $D_\delta \subseteq \Omega$ with $\dist(D_\delta,\partial \Omega) =: \delta$.
	For $\eps < \delta$, let $\psi_\eps = \rho_\eps * \psi$ be the convolution with $\rho_\eps$.
	Then, we obtain
	\begin{eq*}
		\|\psi - \psi_\eps\|_{H(\cL^*,D_\delta)} \to 0 \text{~as~} \eps \to 0.
	\end{eq*}
\end{lem}
\begin{proof}
	The convergence in $L^2(D_\delta)$ follows from standard theory, see, e.g.,~\cite[4.15(2)]{bib:AltNuernberg}.
	Hence, we are left to show that
	\begin{eq*}
		\|\ddiv(\psi F) - \ddiv(\psi_\eps F)\|_{L^2(D_\delta)} \to 0.
	\end{eq*}
	We estimate
	\begin{eq*}
		\|\ddiv(\psi_\eps F) - \ddiv(\psi F)\|_{L^2(D_\delta)}
		&\leq \|\ddiv(\psi_\eps F) - \ddiv((\psi F)_\eps)\|_{L^2(D_\delta)}\\
		&\phantom{\leq}~+ \|\ddiv((\psi F)_\eps) - \ddiv(\psi F)\|_{L^2(D_\delta)}.
	\end{eq*}
	The first term converges to zero according to \cref{lem:comparison}.
	For the second term, we observe by \cref{eq:mollifier:derivative} that $\ddiv(\rho_\eps * (\psi F)) = \rho_\eps * (\ddiv(\psi F))$ and hence
	\begin{eq*}
		\|\ddiv(\psi F) - \ddiv((\psi F)_\eps)\|_{L^2(D_\delta)} \to 0 \text{~as~} \eps \to 0.
	\end{eq*}
	This proves the assertion.
\end{proof}

With these preparations at hand, we are finally ready to prove the density of $C^\infty(\Omega) \cap \pfs$ in $\pfs$.
%
%
\begin{proof}[Proof of \cref{thm:pfs_smooth_density}\,\cref{enum:pfs_smooth_density}]
	Consider the locally finite open covering $(U_k)_{k \in \bN}$ of $\Omega$ in \cref{eq:open_covering_example} with $\bar U_k \subseteq \Omega$ and an associated partition of unity $(\eta_k)_{k \in \bN}$ with $\eta_k \in C_0^\infty(U_k)$ and $\sum_{k \in \bN} \eta_k = 1$ on $\Omega$.
	Take an arbitrary $\eps > 0$.
	Using \cref{lem:aux_apx} we select for all $k \in \bN$ a function
	$\psi_{\eps,k} \in C^\infty(U_k)$ with
	\begin{eq*}
		\|\psi - \psi_{\eps,k}\|_{H(\cL^*, U_k)} \leq \frac{\eps \cdot 2^{-k}}{(\|\eta_k\|_{C^1(\bar \Omega)} + 1)(1 + \|F\|_{L^\infty(\Omega)}^2)^{\frac{1}{2}}}.
	\end{eq*}
	We show that $\eta_k \psi \in \pfs$.
	To see this, take $\varphi \in C_0^\infty(\Omega)$.
	We note that
	\begin{eq*}
		\nabla (\eta_k \varphi) = \eta_k \nabla \varphi + \varphi \nabla \eta_k
	\end{eq*}
	and obtain by partial integration
	\begin{eq}\label{eq:div_chainrule}
		\int_\Omega \eta_k \psi F \nabla \varphi \, \dx
		&= \int_\Omega \psi F (\nabla (\eta_k \varphi) - \varphi \nabla \eta_k) \, \dx\\
		&= -\int_\Omega \ddiv(\psi F) \eta_k \varphi \, dx - \int_\Omega \psi F \nabla \eta_k \varphi \, \dx.
	\end{eq}
	This implies that $\eta_k \psi F \in H(\ddiv, \Omega)$ with $\ddiv(\eta_k \psi F) = \eta_k \ddiv(\psi F) + \psi F \nabla \eta_k$.
	We define $\psi_\eps = \sum_{k \in \bN} \eta_k \psi_{\eps,k}$.
	As in every sufficiently small neighborhood of a point $x \in \Omega$ only finitely many $\eta_k$ are nonzero, we have from the smoothness of $\eta_k$ and $\psi_{\eps,k}$ that $\psi_\eps \in C^\infty(\Omega)$.
	From \cref{eq:div_chainrule}, we deduce
	\begin{eq*}
		\ddiv((\psi - \psi_\eps)F) = \sum_{k \in \bN}\eta_k \ddiv((\psi - \psi_{\eps,k})F) + \sum_{k \in \bN} \nabla \eta_k (\psi - \psi_{\eps,k})F.
	\end{eq*}
	Further, we derive
	\begin{eq*}
		\|\psi - \psi_\eps\|_{L^2(\Omega)}
		&= \left\| \sum_{k \in \bN} \eta_k (\psi - \psi_{\eps,k}) \right\|_{L^2(\Omega)}
		\leq \sum_{k \in \bN} \|\eta_k\|_{L^\infty(\Omega)} \|\psi - \psi_{\eps,k}\|_{L^2(U_k)}\\
		&\leq \sum_{k \in \bN} \|\eta_k\|_{L^\infty(\Omega)} \frac{\eps \cdot 2^{-k}}{(\|\eta_k\|_{C^1(\bar \Omega)} + 1)(1 + \|F\|_{L^\infty(\Omega)}^2)^{\frac{1}{2}}} \leq \eps
	\end{eq*}
	and
	\begin{eq*}
		\|\ddiv((\psi - \psi_\eps)F)\|_{L^2(\Omega)} &\leq \sum_{k \in \bN} \|\eta_k\|_{L^\infty(\Omega)} \|\ddiv((\psi - \psi_{\eps,k})F)\|_{L^2(U_k)}\\
		&\phantom{\leq}~+ \sum_{k \in \bN} \|\nabla \eta_k\|_{L^\infty(\Omega;\bR^d)} \|F\|_{L^\infty(\Omega)} \|\psi - \psi_{\eps,k}\|_{L^2(U_k)}\\
		&\leq (1 + \|F\|_{L^\infty(\Omega)}^2)^{\frac{1}{2}} \sum_{k \in \bN} \|\eta_k\|_{C^1(\bar \Omega)} \|\psi - \psi_{\eps,k}\|_{H(\cL^*, U_k)}\\
		&\leq \eps \cdot (1 + \|F\|_{L^\infty(\Omega)}^2)^{\frac{1}{2}} \sum_{k \in \bN} \frac{2^{-k} \|\eta_k\|_{C^1(\bar\Omega)}}{(1 + \|\eta_k\|_{C^1(\bar \Omega)})(1 + \|F\|_{L^\infty(\Omega)}^2)^{\frac{1}{2}}}\\
		&\leq \eps.
	\end{eq*}
	This proves $\|\psi_\eps\|_{\pfs} \leq \|\psi\|_{\pfs} + \|\psi - \psi_\eps\|_{\pfs} < \infty$ and $\|\psi - \psi_\eps\|_{\pfs} \leq \sqrt{2}\eps$.
	Thus, we have shown that $C^\infty(\Omega) \cap \pfs$ is dense in $\pfs$.
\end{proof}

As we are also interested in boundary values, it is not enough to have the density of smooth functions in the domain.
Therefore, we show next the density of functions that are smooth to the boundary.
%
%
\begin{proof}[Proof of \cref{thm:pfs_smooth_density}\,\cref{enum:pfs_smooth_density_bdry}]
	The proof is split into several steps.
	%
	%
	\paragraph{Step 1.} Consider $x \in \partial \Omega$ with $r > 0$, $Q = Q(x,r)$, $Q' := Q(x, \frac{r}{2})$ such that $\Omega \cap Q = \{y : \gamma(y_1, \dots, y_{d - 1}) < y_d\}$ for a Lipschitz continuous function $\gamma$ in the respective coordinate system.
	Consider a function $\psi \in H(\cL^*, \Omega \cap Q)$ that is zero near $\partial Q' \cap \Omega$.
	Let $\alpha := \Lip(\gamma) + 2$.
	For $y \in Q'$ we define $y^\eps = y + \alpha \eps e_d$ with $e_d$ being the $d$-th unit vector in the coordinate system.
	Then, we get for $y \in \Omega \cap Q'$ that
	\begin{eq*}
		\gamma(y^\eps_1, \dots, y_{d - 1}^\eps) = \gamma(y_1, \dots, y_{d - 1}) < y_d < y_d + \alpha \eps = y_d^\eps
	\end{eq*}
	and for $\eps < \frac{r}{4 \alpha}$ that $|y_j^\eps - x_j| = |y_j - x_j|$ for $j = 1, \dots, d - 1$ and $|y_d^\eps - x_d| \leq |y_d - x_d| + \alpha \eps < \frac{r}{2} + \alpha \frac{r}{2 \alpha} = r$, which implies $y^\eps \in Q$.
	Thus, with $\eps \leq \frac{r}{2\alpha}$ we deduce that for all $y \in \Omega \cap Q'$ we have $y^\eps \in \Omega \cap Q$.
	Moreover, we get $B(y^\eps, \eps) \subseteq \Omega \cap Q$.
	To see this, take $z \in B(y^\eps, \eps)$ and observe
	\begin{eq*}
		\gamma(z_1, \dots, z_{d - 1})
		&\leq \gamma(y^\eps_1, \dots, y_{d - 1}^\eps) + \Lip(\gamma) |z - y^\eps|
		= \gamma(y_1, \dots, y_{d - 1}) + \Lip(\gamma) \eps\\
		&< y_d + \Lip(\gamma) \eps
		< y_d + (\alpha - 1) \eps
		= y_d^\eps - \eps
		< y_d^\eps + (z_d - y_d^\eps)
		= z_d.
	\end{eq*}
	Additionally, we get for $j = 1, \dots, d - 1$ that $|z_j - x_j| \leq |y_j^\eps - x_j| + |z_j - y_j^\eps| = |y_j - x_j| + |z_j - y_j| \leq \frac{r}{2} + \frac{r}{4\alpha} < r$.
	For $j = d$, we get
	\begin{eq*}
		|z_d - x_d| \leq |y_d^\eps - x_d| + |z_d - y_d^\eps| \leq |y_d - z_d| + \alpha \eps + \eps \leq \frac{r}{2} + \alpha \frac{r}{4 \alpha} + \frac{r}{4 \alpha} < r.
	\end{eq*}
	Thus, we deduce $B(y^\eps, \eps) \subseteq \Omega \cap Q$ with $\eps = \frac{r}{4\alpha}$, which is independent of the choice of $y$.
	%
	%
	\paragraph{Step 2.} Take a function $\psi \in H(\cL^*, \Omega \cap Q)$ that vanishes outside $Q'$.
	Let $\widetilde \psi$ be the extension of $\psi$ to $Q$ by zero.
	Set $Q'':= Q(x,\frac{3r}{4})$.
	This yields $\dist(Q'', \partial Q) = \frac{r}{4}$.
	Let the following sequence of functions $\widetilde\psi_\eps \colon Q'' \to \bR$ be defined by
	\begin{eq*}
		\widetilde \psi_\eps(y)
		:= \int_{\bR^d} \rho_\eps(z - y^\eps) \widetilde \psi(z) \, \mathrm{d}z
		= \int_{B(y, \eps)} \rho_\eps(z - y) \widetilde \psi(z + \alpha \eps e_d) \, \mathrm{d}z.
	\end{eq*}
	Since in particular $\widetilde \psi \in L^2(Q'')$, we get $\widetilde \psi_\eps \in C^\infty(Q'')$ for $\eps = \frac{r}{4\alpha} < \frac{r}{4}$.
	Restricting to $\overline{(\Omega \cap Q')}$ guarantees $\psi_\eps := \widetilde \psi_\eps|_{\overline{\Omega \cap Q'}} \in C^\infty(\overline{\Omega \cap Q'})$.
	Moreover, we can show with the same techniques as in the proofs of \cref{lem:comparison,lem:aux_apx} that
	\begin{eq*}
		\psi_\eps \to \psi \text{~in~} H(\cL^*, \Omega \cap Q').
	\end{eq*}
	As $\psi$ vanishes outside $Q'$, we can also guarantee $\psi_\eps$ to vanish outside $Q''$.
	Therefore, we can extend $\psi_\eps$ to $\Omega$ by zero on $\Omega \backslash Q''$ and obtain $\psi_\eps\in C^\infty(\bar \Omega)$.
	%
	%
	\paragraph{Step 3.}
	Since $\partial \Omega$ is a compact set, there exists a finite covering with cubes $Q_k' = Q(x_k, \frac{r_k}{2})$ for $k = 1, \dots, N$ as in Steps 1 and 2.
	Let $(\eta_k)_{k = 0}^N$ be a partition of unity such that $0 \leq \eta_k \leq 1$ for all $k = 1, \dots, N$ on $\Omega$ and $\supp(\eta_k) \subseteq Q_k'$ for $k = 1, \dots, N$ and $\supp(\eta_0) \subseteq \Omega$.
	Take $\eps > 0$ so small that the arguments in the previous steps are applicable.
	For $k = 1, \dots, N$, we apply Step 2 to $\eta_k \psi$ and find $\psi_{\eps, k} \in C^\infty(\bar \Omega)$ with
	\begin{eq*}
		\|\eta_k \psi - \psi_{\eps,k}\|_{H(\cL^*, \Omega)}
		= \|\eta_k \psi - \psi_{\eps,k}\|_{H(\cL^*, \Omega \cap Q_k')} \leq \frac{\eps}{N + 1}.
	\end{eq*}
	As $\eta_0 \psi \in H(\cL^*, \Omega)$ and $\supp \eta_0 \subseteq \Omega$, we can find by \cref{lem:aux_apx} a function $\psi_{\eps,0} \in C^\infty(\bar \Omega)$ with $\| \eta_0 \psi - \psi_{\eps,0}\|_{H(\cL^*, \Omega)} \leq \frac{\eps}{N + 1}$.
	We define
	\begin{eq*}
		\psi_\eps := \sum_{k = 0}^N \psi_{\eps,k}.
	\end{eq*}
	As this sum is finite, we get $\psi \in C^\infty(\bar \Omega)$ and
	\begin{eq*}
		\|\psi - \psi_\eps\|_{H(\cL^*,\Omega)}
		\leq \sum_{k = 0}^N \|\eta_k \psi - \psi_{\eps,k}\|_{H(\cL^*,\Omega)}
		\leq \eps.
	\end{eq*}
	This proves the assertion of \cref{thm:pfs_smooth_density}\,\cref{enum:pfs_smooth_density_bdry}.
\end{proof}

After showing the density of smooth functions in $\pfs$, we address next the subspace $\pfsz$.
In contrast to the previous constructive approach, we aim at an indirect approach as in the proof of \cite[Theorem 2.4]{bib:GiraultRaviart_FEMNavierStokes}.
For this sake, we use the following result.
%
%
\begin{thm}[cf. {\cite[Eq.\ (2.14)]{bib:GiraultRaviart_FEMNavierStokes}}]\label{thm:density_indirect}
	Consider a Banach space $X$ and a subspace $D \subseteq X$.
	Then $D$ is dense in $X$, if and only if all functionals in $X^*$ that vanish on $D$ also vanish on~$X$.
\end{thm}
%
%
\begin{proof}[Proof of \cref{thm:pfs_smooth_density}\,\cref{enum:pfsz_smooth_density}]
	We use \cref{thm:density_indirect}.
	Consider a functional $\xi \in \pfsz^*$ that vanishes on $C_0^\infty(\Omega)$.
	Since $\pfsz$ is a Hilbert space, there exist $u, v \in L^2(\Omega)$ that fulfill
	\begin{eq*}
		\langle \xi, \psi \rangle_{\pfsz^*, \pfsz} = (u,\psi)_{L^2(\Omega)} + (v, \ddiv(\psi F))_{L^2(\Omega)} \text{~for all~} \psi \in \pfsz
	\end{eq*}
	with $\ddiv(u F) = v$.
	Testing with arbitrary $\varphi \in C_0^\infty(\Omega)$, we get
	\begin{eq*}
		0
		&= \langle \xi , \varphi \rangle
		= (u, \varphi)_{L^2(\Omega)} + (v, \ddiv(\varphi F))_{L^2(\Omega)}\\
		&= (u, \varphi)_{L^2(\Omega)} + (v F, \nabla \varphi)_{L^2(\Omega; \bR^d)} + (\ddiv(F) v, \varphi)_{L^2(\Omega)}.
	\end{eq*}
	This implies that $v \in H(\cL^*,\Omega)$ with $\ddiv(v F) = u + \ddiv(F) v$.
	By \cref{thm:pfs_smooth_density}\,\cref{enum:pfs_smooth_density_bdry}, there exists a sequence $(v_n)_{n \in \bN}$ with $ v_n \in C^\infty(\bar\Omega)$ and $v_n \to v$ in $\pfs$.
	Taking $\psi \in \pfsz$ we get by \cref{thm:green_pfs_h1} that
	\begin{eq*}
		(v_n,\ddiv(\psi F))_{L^2(\Omega)}
		= -(\ddiv(v_n F), \psi)_{L^2(\Omega)} + (\ddiv F \psi, w)_{L^2(\Omega)} + \langle \tr_{F \nu}(\psi) , \tr(v_n) \rangle.
	\end{eq*}
	The last term is zero since $\psi \in \pfsz$.
	The application of this formula yields
	\begin{eq*}
		\langle \xi , \psi \rangle &= (u,\psi)_{L^2(\Omega)} + (v, \ddiv(\psi F))_{L^2(\Omega)} = (u,\psi)_{L^2(\Omega)} + \lim_{n \to \infty} (v_n,\ddiv(\psi F))_{L^2(\Omega)}\\
		&= (u,\psi)_{L^2(\Omega)} + \lim_{n \to \infty}\left( -(\ddiv(v_n F),\psi)_{L^2(\Omega)} + (\ddiv(F) \psi, v_n)_{L^2(\Omega)} \right)\\
		&= (u,\psi)_{L^2(\Omega)} - (\ddiv(v F), \psi)_{L^2(\Omega)} + (\ddiv(F) v, \psi)_{L^2(\Omega)} = 0.
	\end{eq*}
	Hence, $\xi$ vanishes on $\pfsz$.
	Together with \cref{thm:density_indirect} this proves the assertion.
\end{proof}

\end{document}